\documentclass[journal,draftclsnofoot,onecolumn]{IEEEtran} 

\usepackage{float}
\usepackage{color}
\usepackage[caption=false,font=footnotesize]{subfig}
\usepackage[pdftex]{graphicx}
\graphicspath{{./Figures/}}
\usepackage{amsmath}
\usepackage{amssymb}
\usepackage{amsthm}
\usepackage{algorithm}
\usepackage{algpseudocode}
\usepackage{threeparttable}
\newtheorem{thm}{Theorem}

\newtheorem{prop}[thm]{Proposition}
\newtheorem{lemma}[thm]{Lemma}
\newtheorem{definition}{{\bf Definition}}
\newtheorem{problem}{{\bf Problem}}
\newtheorem{remark}{\noindent {\bf Remark}}

\DeclareMathOperator*{\argmin}{arg\,min}

\newcommand{\Sinit}{\ensuremath{{\cal S}^\textit{init}}}
\newcommand{\Xinit}{\ensuremath{{\cal X}^\textit{init}}}

\newcommand{\enabled}{\ensuremath{\textit{Enabled}}}
\newcommand{\pr}{\ensuremath{\Pi}}
\newcommand{\ps}{\ensuremath{\textit{PS}}}

\newcommand{\g}{\ensuremath{\textit{G}}}
\newcommand{\dom}{\ensuremath{\textit{dom}}}

\renewcommand{\H}{\ensuremath{{\cal H}}}
\newcommand{\T}{\ensuremath{{\cal T}}}

\renewcommand{\int}{\ensuremath{\mathbb{Z}}}
\newcommand{\nnint}{\ensuremath{\mathbb{Z}_+}}
\renewcommand{\S}{\ensuremath{{\cal S}}}

\newcommand{\U}{\ensuremath{{\cal U}}}
\newcommand{\X}{\ensuremath{{\cal X}}}
\newcommand{\J}{\ensuremath{{\cal J}}}
\newcommand{\W}{\ensuremath{{\cal W}}}
\newcommand{\dss}{\ensuremath{{\cal D}}}
\newcommand{\grid}{\ensuremath{{\textit{Grid}}}}
\renewcommand{\P}{\ensuremath{{\cal P}}}
\renewcommand{\L}{\ensuremath{{\cal L}}}
\newcommand{\Ld}{\ensuremath{{\cal L}_d}}
\newcommand{\trace}{\ensuremath{\textit{Tr}}}
\newcommand{\norm}[1]{\ensuremath{\left|\!\left|#1\right|\!\right|}}
\newcommand{\pathf}{\ensuremath{\textit{Paths}}}
\newcommand{\pathm}{\ensuremath{\textit{Paths}^m}}

\newcommand{\str}{\ensuremath{\textit{Str}}}

\newcommand{\abs}{\ensuremath{\textit{Abs}}}
\newcommand{\nnreals}{\ensuremath{\mathbb{R}_+}}
\newcommand{\reals}{\ensuremath{\mathbb{R}}}
\renewcommand{\u}{\ensuremath{{\bf u}}}

\newcommand{\equivs}{\ensuremath{{\equiv}_S}}
\newcommand{\equivu}{\ensuremath{{\equiv}_U}}
\newcommand{\optcar}{\textsf{OptCAR}}

\newcommand{\C}{\ensuremath{{\cal C}}}
\newcommand{\B}{\ensuremath{{\cal B}}}
\newcommand{\extract}{\ensuremath{\textit{ExtractController}}}
\newcommand{\solve}{\ensuremath{\textit{SolveFiniteGame}}}
\newcommand{\consabs}{\ensuremath{\textit{ConsAbs}}}
\newcommand{\redreach}{\ensuremath{\textit{ReduceReach}}}
\newcommand{\cost}{\ensuremath{\textit{C}}}

\begin{document}
	
\title{Abstraction-Refinement Based Optimal Control with Regular Objectives\footnote{A version of this paper is submitted to the IEEE Transactions on Automatic Control. This version is updated with fixed typo found after the submission.}}
\author{
Yoke~Peng~Leong,~\IEEEmembership{Student Member,~IEEE,} 
Pavithra~Prabhakar,~\IEEEmembership{Member,~IEEE,}%
\thanks{Y. P. Leong is with the Control and Dynamical Systems, California Institute of Technology, Pasadena, CA 91125, USA {ypleong@caltech.edu}}%
\thanks{P. Prabhakar is with the Department of Computer Science, Kansas State University, Manhattan, KS 66506, USA {pprabhakar@ksu.edu}}}


\maketitle


\begin{abstract}
This paper presents an abstraction-refinement method to synthesize control inputs
for a discrete-time piecewise linear system. The controlled system behavior satisfies a finite-word
linear-time temporal objective while incurring minimal cost. 
An abstract finite state weighted transition system is constructed from finite
partitions of the state and input spaces by solving
optimization problems. A sequence of suboptimal controllers is obtained by
considering a sequence of uniformly refined partitions. 
The abstract system satisfies the condition that the cost of the optimal control on the abstract system
provides an upper bound on the cost of the optimal control for the
original system. Furthermore, each suboptimal controller gives trajectories that have the cost upper bounded by the cost of the optimal control on the corresponding abstract system. In fact, the costs achieved by the sequence of suboptimal controllers converge to the
optimal cost for the piecewise linear system. 
The tool \optcar ~implements the abstraction-refinement algorithm. 
Examples illustrate the feasibility of this approach to synthesize automatically
suboptimal controllers with improving optimal costs.
\end{abstract}

\begin{IEEEkeywords}
Optimal control, Hybrid systems, Formal methods, Abstraction-refinement.
\end{IEEEkeywords}

\section{Introduction}
Formal synthesis is a paradigm for designing controllers
automatically which are correct-by-construction, and thus, reduces the verification overhead. 
In this paradigm, a mathematical model of a system to be controlled and formal specifications of properties that are expected of the controlled system are given as inputs to compute a controller that ensures the controlled system satisfies the properties. For instance, given a model for the behavior of a robot, synthesize a plan that reaches a given part of the workspace while avoiding certain obstacles.

Since the work of \cite{church63} on automated synthesis, multiple directions are pursued including synthesizing
finite state systems with respect to temporal logic objectives
\cite{manna81,clarke81} and controlling discrete event systems
\cite{ramadge87}.
Early works in hybrid control systems focused on
identifying subclasses of systems for which controller synthesis is decidable including timed automata~\cite{asarin94}, rectangular hybrid
automata~\cite{henzinger99} and o-minimal
automata~\cite{bouyer10,vladimerou11}. 
However, these classes of systems have limited continuous and discrete dynamics, and the synthesis problem becomes undecidable for a relatively simple class of hybrid systems \cite{henzinger95}.

For systems with complex dynamics, \cite{raischieee98} introduced an abstraction based controller
synthesis. Given a system, an abstract model, often a finite state system, is
constructed such that a controller for the
abstract model can be refined into a controller for the given
system. The controller for the abstract model is constructed using results from automata theory, and it is then implemented in the given
system. This method has been successfully applied in controller synthesis of switched
systems~\cite{wong12,liu13}, and in robotic path
planning~\cite{belta07,gazit13}. 

Often, in addition to designing a correct controller,
an application may require optimality condition.
For instance, a robot should reach a desired state with minimum battery. In this paper, we investigate an abstraction-refinement approach to
synthesize optimal controller with regular
properties that allow for specifications such as reaching a target region
or traversing a sequence of regions. A regular property is specified as a (possibly) infinite set of finite traces interpreted
as the allowed behaviors of the system, and generated by a finite
state automaton.
The foundation of our abstraction-refinement procedure and its
correctness rely on defining appropriate preorders on the class of
hybrid systems which preserve the optimal cost. This paper shows that the preordering defined satisfies the fact that if a system $\H_2$ is higher up in the
ordering than a system $\H_1$, then the cost of the optimal controller
for $\H_1$ is at most that of $\H_2$.

In our approach, first, an ``abstraction'' --- a simplified finite state system
--- is constructed from partitions of the state and input spaces, and
the edges of the system are annotated with weights which
over-approximate the costs in the original system. 
Then, a two player game on the finite state system is solved to obtain a
controller for the abstraction, and subsequently, a controller for the
original dynamical system. This approach iteratively consider finer partitions of the state and input
spaces, corresponding to grids of size $C/2^i$ for some constant $C$
and $i = 0, 1, 2, \ldots$. In fact,
for discrete-time piecewise linear systems, if the cost function is continuous and the optimal
control for the original system is robust with respect to the initial states, the cost of
the sequence of controllers constructed converges to the optimal cost. 

\subsection{Related work and contributions}
The main contribution of this paper is the optimality guarantee on the controllers synthesized --- an important missing piece in most previous works that study optimal controller synthesis using formal approaches \cite{Seatzu20061188,AydinGol201578}.
A hierarchical optimal controller synthesis problem was studied \cite{Seatzu20061188}, yet, no formal guarantees on
the optimal cost are provided. Similarly, \cite{AydinGol201578} considered optimal control synthesis by combining linear temporal logic, potential functions and model predictive control without formal guarantees on the optimal cost. On the other hand, the sequence of controllers constructed by our approach converges to the optimal cost for discrete time piecewise linear systems. Furthermore, for each suboptimal controller, the resulting trajectories have cost no greater than the optimal cost of the corresponding abstract system. Hence, when computational resources is limited, the best suboptimal controller found is guaranteed to generate trajectories with known bounded costs. 

In addition, this technique is more general than classical finite horizon optimal control problems \cite{bemporad_ochybrid_2000,lincoln02} because the time horizon is not fixed a priori. Our approach focuses on finite horizon optimal control problems, but the input sequence length is not a priori fixed because the regular property consists of finite traces whereby the length is variable. Apart from that, our method allows for a larger class of cost functions in comparison to previous works \cite{girard2012controller,reissig2013abstraction,mazo2011symbolic}. These works \cite{girard2012controller,reissig2013abstraction,mazo2011symbolic} used abstraction-based methods to find an optimal time controller that gives the shortest path which satisfies certain reachability conditions. In contrast, our method encodes transition cost in the abstraction scheme and thus, allowing for picking a path that is ``shortest'' with respect to a more general class of optimality conditions. Lastly, the technique presented in this paper does not place any prior restriction on the structure of the controllers \cite{karaman08,wolff13}. In \cite{karaman08,wolff13}, trajectory based optimization is applied for
synthesizing optimal control for discrete-time non-linear systems,
however, it constrains the class of control strategies considered (to
either finite paths or lassos). 

The generality of our approach enables control engineers to synthesize controllers with more flexible structure and cost considerations. The method introduced in this paper applies to the general class of discrete-time hybrid systems. However, due to its' generality, the computation burden could be high because
the optimizations that compute the weights
depend on the cost function and the dynamics. 
A prototype tool \optcar \, that implements the abstraction refinement algorithm
is presented (It will be made available for download when the paper is published). It is used to synthesize a (finite) sequence of
controllers for a discrete-time linear system and a linear
piecewise system with reachability objective.

A preliminary version of this work appeared in \cite{OptCAR_ACC}. This extended version provides a more complete discussion of the technique including generalization of Theorem \ref{thm:main_converge} with complete proof and an extra example to show that the resulting controller is similar to the linear quadratic regular for a linear system.

\subsection{Paper Outline}
The rest of this paper is organized as follows: Section
\ref{sec:preliminaries} presents useful mathematical
notations and definitions. Section \ref{sec:wts} defines the weighted transition system and its relevant
concepts, and Section \ref{sec:pre} explains the preorders for
optimal control. The abstraction and refinement of a weighted
transition system is developed in Section
\ref{sec:ar}. Section \ref{sec:pwa} presents the problem formulation for optimal control of
piecewise linear systems, the refinement procedure ($\optcar$) and the
cost analysis. The value iteration scheme for computing an
optimal strategy for a finite transition system is described in Section \ref{sec:fts}. 
Section \ref{sec:example} presents the implementation of $\optcar$ using two examples. 
Lastly, Section \ref{sec:conclusion} summarizes the
paper and states future directions of this work. Some proofs are
provided in the Appendix at the end of this paper for ease of reading. 

\section{Notations} \label{sec:preliminaries} 

The sets of real numbers, non-negative real numbers, integers and
non-negative integers are represented as $\reals$, $\nnreals$, $\int$
and $\nnint$, respectively.
The set of integers $\{0, \ldots, k\}$ is written as $[k]$ and a
sequence $x_0, \ldots, x_k$, denoted as $\{x_i\}_{i \in [k]}$. 

If $M \in \mathbb{R}^{n \times m}$ is a matrix, $\norm{M}_\infty =
\max_i \sum_{j=1}^m |M_{ij}|$.
If $z  = (z_0, \ldots, z_{k}) \in \mathbb{R}^{k+1}$ is a vector,
$\norm{z}_{\infty} = \max_t ~ \{|z_t|\}_{t\in[k]}$. 

An $\epsilon$-ball around $x$ is defined as $\B_{\epsilon}(x) = \{x'\in
\mathbb{R}^n \mid $ $ \norm{x - x'} _{\infty}$ $ \leq \epsilon\}$. 
Let $S \subseteq \mathbb{R}^{k}$ be a $k$-dimensional subset. 
The function $\grid$ splits $S$ into rectangular sets with $\epsilon$ width. 
That is, $\grid(S, \epsilon) =$
	\begin{equation*}
\left\{ S' \left\arrowvert  \exists d_1, \ldots, d_k \in
        \int, S' = S \bigcap \prod_{i=1}^k (d_i\epsilon, (d_i + 1) \epsilon)\right.\right\}.
	\end{equation*}

Given a function, $f:\mathcal{A} \to \mathcal{B}$, for any $A
\subseteq \mathcal{A}$, $f(A) = \{f(a) | a \in A\}$. The
domain of a function $f$ is denoted as $\dom(f)$. Given an
equivalence relation $R \subseteq A \times A$ and an element 
$a \in A$, $[a]_R =  \{b \,|\, (a, b) \in R\}$ denotes the equivalence
class of $R$ containing $a$.


\section{Weighted Transition Systems} \label{sec:wts}
This section defines a semantic model for discrete time hybrid systems with
cost (i.e., weighted transition systems) and formalizes the optimal
control problem.
\begin{definition}
A weighted transition system is defined as $\T = (\S, \Sinit, \U, \P, \Delta, \L, \W)$, where:
	\begin{itemize}
		\item $\S$ is a set of states;
                  \item $\Sinit \subseteq \S$ is a set of initial states;
		\item $\U$ is a set of control inputs;
                  \item $\P$ is a set of propositions;
		\item $\Delta \subseteq \S \times \U \times \S$ is a
                  transition relation;
                  \item $\L: \S \to \P$ is a state labeling
                    function, and 
		\item $\W: \S \times \U \times \S \to \nnreals$ is the transition cost function.
	\end{itemize}
\end{definition}
Note that an equivalent definition for the proposition set and the labeling function would be to let $\P'$ be a set of propositions and define a labeling function that maps the states onto the power set of $\P'$. The proposition set $\P$ is related to $\P'$ whereby $\P = \{(p_0,\ldots,p_m)\mid p_i \in \P' ~\forall i\in [m]\}$. Current definition of proposition set is chosen for notational simplicity.
In the sequel, a weighted transition
system is referred as a transition system.
For any $s \in \S$, define the set $\enabled(s) = \{u \in \U \mid
\exists s' \in \S ~s.t. ~ (s,u,s') \in \Delta\}$ to represent all inputs that do not transitions the state $s$ out of the predefined set $\S$. 
A transition system is \emph{finite} if $\S$ and $\U$
are finite. A finite state \emph{automaton} (denoted $(\T, P_f)$) is a finite transition system $\T$
along with a proposition $P_f \in \P$ which represents the final
states.
For the rest of the section, fix the transition system $\T = (\S, \Sinit, \U, \P, \Delta,$ $
\L, \W)$. 

\paragraph{Paths and traces}
A \emph{path} of the transition system $\T$ is a sequence of
states and inputs, $\zeta = s_0 u_0 s_1 u_1 s_2 \ldots$,
where $s_0 \in \Sinit$, $s_i \in \S$, $u_i \in \U$, and $(s_i,u_i, s_{i+1})
\in \Delta$. 
The set of all finite paths of $\T$ is denoted $\pathf(\T)$. 
A trace of a transition system is the sequence of state labels of a
path.
The trace of $\zeta$, denoted $\trace(\zeta)$, is the sequence $\L(s_0)
\L(s_1) \ldots$.

\paragraph{Properties}
This paper focuses on linear time properties over finite
behaviors of systems.
A \emph{property} $\pr$ over a set of propositions $\P$ is a set of finite
sequences $\pi = p_0 p_1 \ldots p_k$, where each $p_i \in \P$.
A property describes the desired behaviors of the system.

A property is \emph{regular} if it consists of the traces of a finite
state automaton $(\T, P_f)$, that is, it is the set of all traces of
paths of $\T$ which start in an initial state and end in a state
labelled by $P_f$.
This paper considers regular property that is
specified by a finite state automaton $(\T, P_f)$. Figure \ref{fig:proposition_example} shows an illustration.
The properties expressed by popular logics such
as finite words linear-time temporal logic (LTL) are regular, but their
translation into the finite transition system representation can lead
to an exponential blow up in the number of states with respect to the
size of the formula~\cite{vardi86}. 

\begin{figure}[b!]
      \centering
      \subfloat[Automaton representing a regular property $\Pi$ of a finite behavior.]{\includegraphics[trim = 0mm 0mm 0mm 0mm, clip, width=0.5\linewidth]{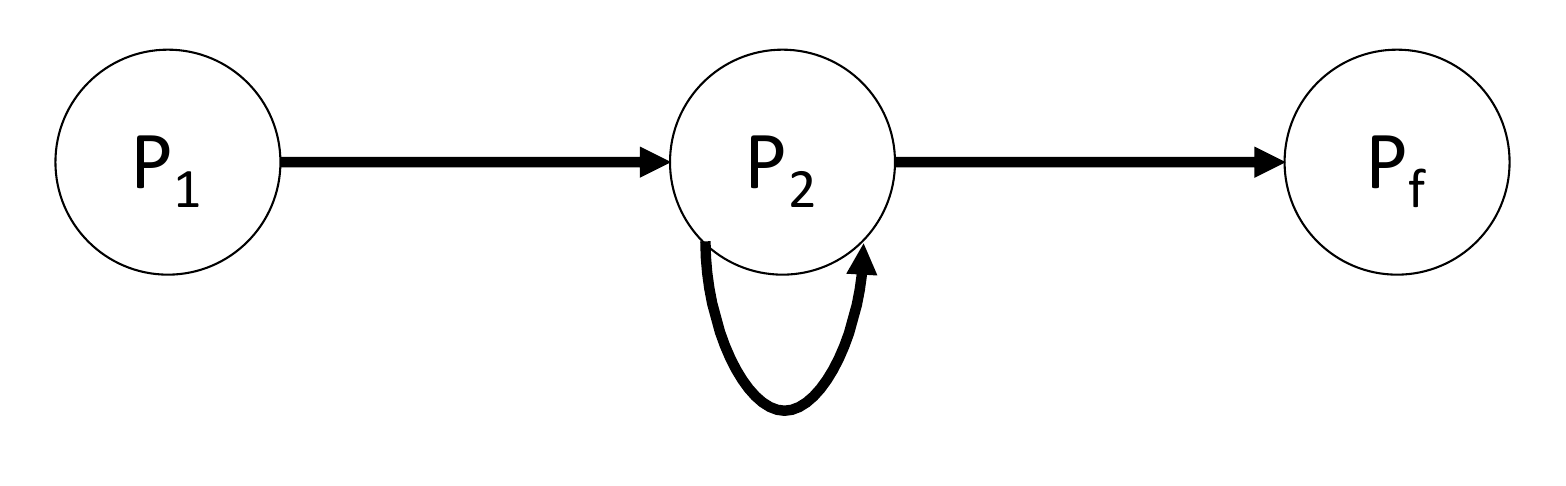}\label{fig:proposition_example_automaton}} 
	  \subfloat[Example paths given by a winning strategy $\sigma$ with respect to $\Pi$.]{\includegraphics[trim = 0mm 0mm 0mm 0mm, clip, width=0.5\linewidth]{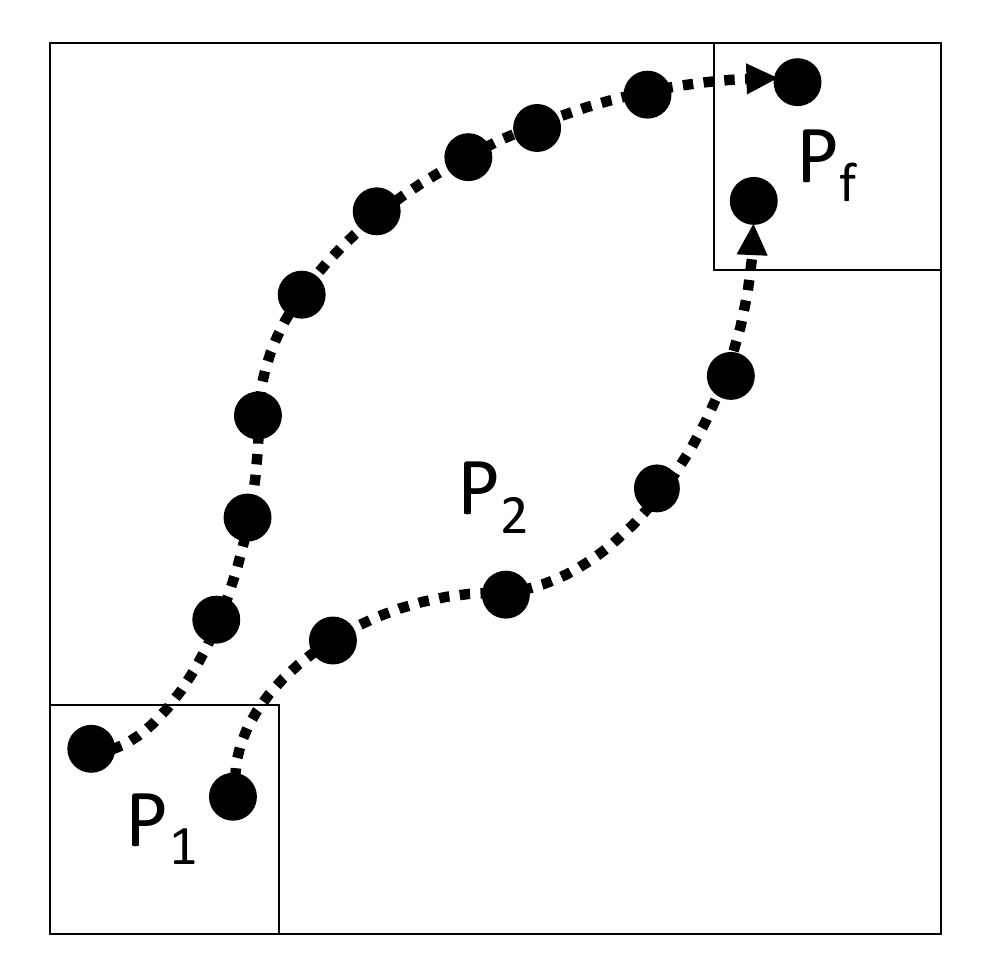}\label{fig:proposition_example_dynamics}}
	  \caption{An example of propositions and corresponding paths given by a winning strategy. This property is more general than a typical finite horizon control problem because the length of the sequence/path is not set a priori.}\label{fig:proposition_example}
\end{figure}

Here, we consider regular properties that specify properties about
finite behaviors as opposed to $\omega$-regular properties that
specify properties about infinite behaviors.
While our framework extends in a natural fashion to $\omega$-regular
properties, the results related to the convergence discussed in
Section \ref{sec:pwa} do not extend to the $\omega$-regular case.
In a separate future work, we intend to explore additional constraints
for the synthesis problem with $\omega$-regular specifications that
will ensure convergence of the control cost to the optimal value. 
Note that regular specifications already capture properties that are
more general than finite horizon control, since a regular property can characterize behaviors involving unbounded length. For instance,
consider reaching a target region without a priori bound on the
number of steps required to reach the region.

\paragraph{Strategies}
A strategy specifies the control inputs to a transition system at
different time points.
More specifically, a \emph{strategy} $\sigma$ for the transition system
$\T$ is a partial function $\sigma: \pathf(\T)
\to \U$ such that for a path $\zeta = s_0 u_0 s_1 \ldots u_{i-1} s_i$,
$\sigma(\zeta) \in \enabled(s_i)$.
A path $\zeta =  s_0 u_0 s_1 u_1 s_2 \ldots$ is said to
\emph{conform} to a strategy $\sigma$, if for all $i$, $\sigma(s_0u_0
\ldots s_i) = u_i$.  

A finite path $\zeta = s_0 u_0 \ldots s_k$ maximally conforms to a strategy
$\sigma$, if $\zeta$ conforms to $\sigma$ and there is no extension
$\zeta' = s_0 u_0 \ldots s_k u_k s_{k+1}$ of $\zeta$ which conforms to
$\sigma$. 
Let $\pathm_\sigma(\T, s_0)$ denotes the maximally conforming finite
paths of $\T$ with respect to $\sigma$ starting at a state $s_0$. 
Let $\str(\T)$ denote the set of all strategies which have no infinite
paths conforming to them.
Note that the length of the paths which conforms to a strategy in
$\str(\T)$ could still be variable.

To synthesize a strategy for 
$\T$ from a state $s_0 \in
\Sinit$ such that all maximal executions conforming to it reach a
state in $\S_f \subseteq \S$,
label the states in $\S_f$  with a unique proposition. Then,
let the property $\pr$ be the set of all traces
corresponding to paths which start in $\Sinit$ and end in $\S_f$,
and do not visit $\S_f$ in the middle.

\begin{definition}
A strategy $\sigma$ for the transition system $\T$ and an initial state
$s_0 \in \S$ is \textbf{winning} with respect to property
$\pr$ over the propositions $\P$, if $\sigma \in \str(\T)$ and
$\trace(\pathm_\sigma(\T,s_0)) \subseteq \pr$. 
\end{definition}

\paragraph{Cost of strategies}
The cost of a path is the sum of the weights on
the individual edges. 
Given a path $\zeta = s_0 u_0 s_1 \ldots$, define: 
\begin{equation*}
\W(\zeta) = \sum_{j} \W(s_j, u_j, s_{j+1}).
\end{equation*}
Consequently, the following proposition holds. 
\begin{prop} \label{prop:monotonic_cost}
	Given $\zeta = s_0 u_0 s_1 \ldots s_k$ and $\zeta' = s'_0 u'_0 $ $s'_1 \ldots s'_k$, if $\W(s_j, u_j,s_{j+1}) \leq \W(s'_j, u'_j,s'_{j+1})$ for all $j$, then $\W(\zeta) \leq \W(\zeta')$.
\end{prop}
This monotonicity property seems trivial, but plays an important role in the analysis later. In fact,
results in this paper carry over for several other cost functions for paths such
as average weight and maximum weight. Over infinite paths, average cost, maximum cost or discounted sum are more natural. 
Nonetheless, the analysis only relies on the fact that the cost of a path is monotonic with respect to the cost on the transitions. For simplicity, we fix one of the definitions.


The cost of a strategy $\sigma$ of the transition system $\T$ with respect to an initial
state $s_0$ is defined as
\[\W(\T, \sigma, s_0) = \sup \{\W(\zeta) \,|\, \zeta \in \pathm_\sigma(\T, s_0)\}.\]
Accordingly, given a property $\pr$ over $\P$, the optimal cost of winning $\T$
from an initial state $s_0$ with respect to a property $\pr$ is
defined as
\begin{equation*}
\W(\T, s_0, \pr)= \inf \{\W(\T, \sigma,s_0) \,|\, \sigma \in \str(\T),
  \trace(\pathm_\sigma(\T, s_0))
  \subseteq \pr\}.
\end{equation*} 
The cost is taken to be infinity if the minimization is over an empty
set. 
Denote an optimal strategy that achieves the optimal cost as
$\sigma{(\T, s_0, \pr)}$. Note that the optimal strategy may not be unique or
exist.

\paragraph{Optimal control problem}
Given the transition system $\T$, an initial state $s_0$ and a property
$\pr$, the optimal control problem is to compute 
an optimal winning strategy from $s_0$ with respect to $\pr$, if it
exists, and the optimal cost of winning.

\section{Preorders for optimal control} 
\label{sec:pre}

In this section, a preorder on the class of transition
systems is defined such that it preserves the optimal cost of winning. In other words,
the optimal cost of winning in a system higher up in the ordering is
an upper bound on the optimal cost of winning in a system below it.
For this, the definition of alternating
simulations~\cite{alur98alt} is extended to include costs. 


\begin{definition}\label{def:simulation}
Given two transition systems $\T_i = (\S_i,\Sinit_i,$ $\U_i,\P, \Delta_i,
\L_i, \W_i)$, for $i = 1, 2$, a simulation from $\T_1$ to $\T_2$ is a
pair of relations $(\alpha, \beta)$, where $\alpha \subseteq \S_1
\times \S_2$ and $\beta \subseteq  \S_1 \times \U_1 \times \S_2 \times
\U_2 $, such that: 
	\begin{enumerate}
		\item $\forall~(s_1,s_2)\in\alpha$, $\L_1(s_1) = \L_2(s_2)$.
		\item $\forall~s_1 \in \S_1^{init}$, $\exists~s_2 \in
                  \S_2^{init}$ such that $(s_1, s_2) \in \alpha$;
		\item $\forall~(s_1, s_2)\in \alpha$ and $u_2
                  \in Enabled(s_2)$, $\exists~u_1 \in Enabled(s_1)$ such that: 
			\begin{enumerate}
				\item $(s_1, u_1, s_2, u_2)\in \beta$
				\item $\forall~(s_1, u_1, s'_1) \in
                                  \Delta_1$, $\exists~(s_2, u_2, s'_2)
                                  \in \Delta_2$ such that $(s'_1,
                                  s'_2) \in \alpha$ and $\W_1(s_1,
                                  u_1, $ $ s'_1) \leq \W_2(s_2, u_2, s'_2)$.
			\end{enumerate}	
	\end{enumerate}
\end{definition}
Let $\T_1 \preceq_{(\alpha, \beta)} \T_2$ denote that
$(\alpha,\beta)$ is a simulation from $\T_1$ to $\T_2$. If
there exists some $(\alpha, \beta)$ such that $\T_1 \preceq_{(\alpha,
  \beta)} \T_2$, then $\T_2$ simulates $\T_1$, and it is denoted as $\T_1 \preceq \T_2$. 

\begin{thm}
\label{thm:pre}
$\preceq$ is a preorder on the class of transition systems over a set
of propositions $\P$.
\end{thm}
\begin{IEEEproof}
	Define $(\alpha, \beta)$ to be identity relations on the state
        and input spaces, then
        $\T \preceq_{(\alpha, \beta)} \T$, and
          hence $\preceq$ is reflexive. To show that $\preceq$ is
          transitive, suppose $\T_1 \preceq_{(\alpha_1,\beta_1)} \T_2$
          and $\T_2 \preceq_{(\alpha_2,\beta_2)} \T_3$. Define
          $\alpha$ such that $(s_1,s_3) \in \alpha$ if $(s_1,s_2)\in
          \alpha_1$ and $(s_2,s_3) \in \alpha_2$ for some $s_2$, and
          define $\beta$ such that $(s_1,u_1,s_3,u_3) \in \beta$ if
          $(s_1,u_1,s_2,u_2)\in \beta_1$ and $(s_2,u_2,s_3,u_3) \in
          \beta_2$ for some $(s_2, u_2)$. Then, $\T_1 \preceq_{(\alpha,\beta)} \T_3$. 
\end{IEEEproof}
The next result shows that $\preceq$ is an ordering on the transition
systems which ``preserves'' optimal control.
%
\begin{thm}
\label{thm:preserve}
Given two transition systems $\T_i = (\S_i,\Sinit_i,$ $ \U_i, \P, \Delta_i,
\L_i, \W_i)$ for $i = 1,2$, let $\pr$ be a property over a
set of propositions $\P$, $\T_1 \preceq_{(\alpha,\beta)} \T_2$ and
$(s_0, s'_0) \in \alpha$ for $s_0\in \Sinit_1$ and $s'_0 \in
\Sinit_2$. 
If there exists a winning strategy $\sigma_2$ for $\T_2$ from $s'_0$
with respect to $\pr$, then there exists a winning strategy $\sigma_1$
for $\T_1$ from $s_0$ with respect to $\pr$ such that $\W_1(\T_1,
\sigma_1,s_0) \leq \W_2(\T_2, \sigma_2,s'_0)$.
Hence, $\W_1(\T_1, s_0,\pr) \leq \W_2(\T_2, s'_0,\pr)$.
\end{thm}

\begin{IEEEproof}
Let $\sigma_2$ be a strategy for $\T_2$ and $s'_0$.
In addition, define a partial mapping $\g:\pathf(\T_1) \to
\pathf(\T_2)$ such that the domain of
$\g$ is the set of all paths from $s_0$ that conform to $\sigma_1$, and for any path
$\zeta_1$ in the domain of $\g$, $\L_1(\zeta_1)  = \L_2(\g(\zeta_1))$, and
$\W_1(\zeta_1) \leq \W_2(\g(\zeta_1))$.
This construction ensures that if $\sigma_2$ is winning from $s'_0$
with respect to $\pr$, then so is $\sigma_1$ from $s_0$ and
$\W_1(\T_1, \sigma_1,s_0) \leq \W_2(\T_2, \sigma_2,s'_0)$.
We also ensure that if $\g(\zeta_1) = \zeta_2$, then
$(s_k, s'_k) \in \alpha$,
where $s_k$ and $s'_k$ are the end states of $\zeta_1$ and $\zeta_2$,
respectively.
Further, for any $\zeta_1$ in the domain of $\g$, $\zeta_1$ is a maximal
path conforming to $\sigma_1$ if and only if $\zeta_2$ is a maximal
path conforming to $\zeta_2$.

Next, define $\sigma_1$ and $\g$ by induction on the length of
words in their domain.
Set $\g(s_0) = s'_0$.
Suppose $\sigma_1$ for paths of length $k-1$ and $\g$ for
paths of length $k$, are defined such that the invariant holds.
Let $\zeta_1 = s_0u_0s_1 \ldots s_k$ conform to
$\sigma_1$. Then, $\g(\zeta_1)$ is defined.
Let $\g(\zeta_1) =  s'_0u'_0s'_1 \ldots s'_k$ and $(s_k, s'_k) \in
\alpha$.
If $\g(\zeta_1)$ is a maximal path conforming to $\sigma_2$, then $\sigma_1(\zeta_1)$ is not defined (i.e.,
$\zeta_1$ is not in the domain of $\sigma_1$).
Otherwise $\sigma_2(\g(\zeta_1)) = u'_{k}$.
Then, from the second condition of simulation, there
exists $u_{k}$ such that $(s_k, u_{k}, s'_k, u'_{k}) \in \beta$.
Choose $\sigma_1(\zeta_1) = u_{k}$.
For any $\zeta_2 = s_0u_0s_1 \ldots s_{k+1}$,
define $\g(\zeta_2) = s'_0u'_0s'_1 \ldots s'_{s+1}$ such
that $(s_{k+1}, s'_{k+1}) \in \alpha$ and $\W_1(s_k,u_{k},s_{k+1})$
$\leq \W_2(s'_k,u'_{k},s'_{k+1})$. 
It can be verified that the construction satisfies the inductive
invariant.
\end{IEEEproof}

\section{Abstraction/Refinement} \label{sec:ar}
In this section, the abstraction refinement procedure for constructing finite state
systems which simulate a given transition system is presented.
The state and input spaces are divided into finite
number of parts, and they are used as symbolic states and inputs,
respectively, in the abstract transition system. Henceforth, fix a transition system $\T = (\S,\Sinit,\U, \P, \Delta, \L,$ $
 \W)$. 

\begin{definition}
	A transition system $\T$ is a complete transition system if for all $s \in \S$, $\U = Enabled(s)$.
\end{definition}

\subsection{Abstraction} \label{sec:abstraction}

An abstraction function constructs an abstract transition
system $\abs(\T, \equivs, \equivu)$ given the transition
system $\T$, and two equivalence
relations $\equivs$ and $\equivu$ on the state-space $\S$ and the
input-space $\U$, respectively.
To ensure a well defined abstract transition system, $\equivs$ on $\S$ needs to respect both the set of labels $\L$ and the set of initial states $\Sinit$. In other words, the labels are the same for all equivalent states, and the initial states in the set $\Sinit$ are not equivalent to any states outside of the set $\Sinit$. More formally, an equivalence relation $\equivs$ on $\S$ respects $\L$, if for all $(s_1, s_2) \in \equivs$, $\L(s_1) = \L(s_2)$. Furthermore, an equivalence relation $\equivs$ on $\S$ respects $\Sinit$, if for all $(s_1, s_2) \in \equivs$ where $s_1 \in \Sinit$, $s_2 \in {\Sinit}$. 

\begin{definition} \label{def: abstraction}
Let $\equivs \subseteq \S \times \S$ and $\equivu \subseteq \U \times
\U$ be two equivalence relations of finite index such that $\equivs$
respects the labeling function $\L$ and the initial states $\Sinit$. 
$\abs(\T, \equivs, \equivu) = (\S',{\Sinit}', \U', \P, \Delta', \L',$ $ \W')$, where:
\begin{itemize}
	\item $\S' = \{ [s]_{\equivs} \mid s \in \S\} $ is the equivalence classes of $\equivs$.
	\item ${\Sinit}'= \{ [s]_{\equivs} \mid s \in {\Sinit}\} \subseteq \S'$ .
          \item $\U' = \{ [u]_{\equivu} \mid u \in \U\}$ is the equivalence classes of $\equivu$.
            \item $\Delta' = \{ (S_1, U, S_2) \mid  \exists s \in S_1, s'\in S_2, u \in U, ~ s.t. ~ (s, u, s') \in \Delta  \}$.
              \item
                For $S \in \S'$, $\L'(S) = \L(s)$ for any $s \in S$. 
                \item
                  For $(S_1, U, S_2) \in \Delta'$, $\W'(S_1, U, S_2) = \sup \{\W(s_1, u,$
                  $ s_2) \mid  s_1 \in S_1,s_2 \in S_2, u \in U, (s_1, u, s_2) \in \Delta\}.$
\end{itemize}
\end{definition}
Call $\T$ the concrete system and $\abs(\T, \equivs, \equivu)$
the abstract system. 
The next proposition states that the abstract system simulates the
concrete system.
\begin{prop}
\label{prop:abs}
If $\T$ is a complete transition system, $\T \preceq \abs(\T, \equivs, \equivu)$.
\end{prop}

\begin{IEEEproof}
	Define $\alpha = \{ (s,[s]_{\equivs}) \,|\, s \in \S \}$, and
  	$\beta = \{(s,u,[s]_{\equivs},$ $[u]_{\equivu}) \,|\, s \in
        \S$ and $u\in \U\}$. 
        Then, properties in Definition
	\ref{def:simulation} are satisfied. 
	
	Consider $\abs(\T, {\equivs}, \equivu) = (\S',{\Sinit}', \U', \P, \Delta', \L',
	        \W')$ as in Definition \ref{def: abstraction}. Define $(s,[s]_{\equivs}) \in \alpha$ for $s \in \S$, and
	        $(s,u,[s]_{\equivs},[u]_{\equivu}) \in \beta$ for $s \in \S$ and
	        $u\in \U$. 
			
			The first property in Definition
	        \ref{def:simulation} is satisfied by construction because for all $S \in \S'$, $\L'(S) = \L(s)$ for any $s \in S$. The second property also holds by construction of ${\Sinit}'$ where $\forall s \in \Sinit$, there exists a $[s]_{\equivs} \in {\Sinit}'$ such that $(s,[s]_{\equivs}) \in \alpha$.
			
			To verify the third property, consider any $(s_1,S_1) \in
	        \alpha$ and $U \in Enabled(S_1)$. Because $U \in \U'$, there exists a $u \in \U$ where
	        $[u]_{\equivu} = U$. Given that $Enabled(s_1) = \U$ for a complete transition system, $u \in Enabled(s_1)$. By definition of $\beta$, $(s, u, S_1, U) \in
	        \beta$. Furthermore, for $(s_1, u, s_2) \in \Delta$, there exists $S_1 \in \S'$ such that $(s_1, S_1) \in \equivs$, $S_2 \in \S'$ such that $(s_2, S_2) \in \equivs$, and $U \in \U'$ such that $(u, U) \in \equivu$. Furthermore, by construction, $(S_1,U,S_2) \in \Delta'$ if $(s_1, u, s_2) \in \Delta$. Thus, there exists a $(S_1,U,S_2) \in \Delta'$ where $(s_2,S_2) \in \alpha$. Lastly, $\W'(S_1, U, S_2) \geq \W(s_1, u, s_2) $
	        because $\W'$ is the maximum over all $s_1 \in S_1, u \in U,
	        s_2 \in S_2$ of $\W(s_1, u , s_2)$.
\end{IEEEproof}

\subsection{Refinement}
We can construct a sequence of abstract systems which are closer to
the original system than their predecessors in the sequence, by
choosing finer equivalence relations on the state and input spaces.

\begin{definition}
Let $\T_1$ and $\T_3$ be transition systems such that $\T_1 \preceq
\T_3$. 
A transition system $\T_2$ is said to be a refinement of $\T_3$ with
respect to $\T_1$, if $\T_1 \preceq \T_2 \preceq \T_3$.
\end{definition}

\begin{prop} 
\label{thm:refinement}
Let $\equivs, \equivs' \subseteq \S \times \S$ and $\equivu, \equivu'
\subseteq \U \times \U$ be equivalence relations of finite index such
that $\equivs' \subseteq \equivs$ and $\equivu' \subseteq \equivu$. 
Then, $\abs(\T, \equivs', \equivu')$ is a refinement of $\abs(\T, \equivs, \equivu)$ with respect to $\T$.
\end{prop}

\begin{IEEEproof}
First, $\T \preceq \abs(\T, \equivs', \equivu')$ follows from
Proposition \ref{prop:abs}.
	Define $\alpha = \{([s]_{{\equivs}'}, $ $[s]_{{\equivs}})
        \,|\, s \in \S\}$ and $\beta = \{ ([s]_{{\equivs}'},
        [u]_{\equivu'},$ $[s]_{{\equivs}},$ $ [u]_{\equivu}) \,|\, s
        \in \S$ and $u \in \U\}$. Then, properties in Definition
	\ref{def:simulation} are satisfied for $\abs(\T,$ $ \equivs',
        $ $\equivu') \preceq_{(\alpha, \beta)} \abs(\T,$ $ \equivs,
        \equivu)$, and thus, $\T \preceq \abs(\T,$ $ \equivs', $
        $\equivu') \preceq \abs(\T,$ $ \equivs, \equivu)$. 
\end{IEEEproof}

\section{Optimal Control of Piecewise Linear Systems} \label{sec:pwa}

This section considers an optimal control problem for discrete-time piecewise linear systems.
The abstraction refinement approach is applied to construct a series of controllers with improving suboptimal costs that converge to the optimal cost under the existence of a robust optimal control.

\subsection{Problem Formulation}

A discrete-time piecewise linear system is a tuple $(\X, \Xinit, \U, \P,$ $\{(A_i, B_i, P_i)\}_{i \in [m]},$ $ \Ld, \J)$, where the state-space $\X \subseteq \reals^n$ and the input-space $\U \subseteq \reals^p$  are compact sets, $\Xinit \subseteq \X$ is the set of initial states, $\P$ is a finite set of propositions, $A_i \in \reals^{n \times n}, B_i \in \reals^{n \times p}$ and $P_i$ is a polyhedral set, such that $\{P_i\}_{i\in m}$ is a polyhedral partition of $\X$, $\Ld: [m] \to \P$ is a labeling function and $\J: \X \times \U \to \nnreals$ is a continuous cost function. Note that $A_i$ and $B_i$ can be the same for different $i$.
We associate a unique label to each region $P_i$. We could have assigned different labels to different regions in some polyhedral partition of $P_i$; we do not lose expressiveness here, since, the latter can be transformed to the former problem by considering a finer partition whose regions are the regions partitioning each $P_i$ according to the label. 

Given an initial state $x_0 \in \Xinit$ and a sequence of control inputs $\u = \{u_t\}_{t \in [k]}$, where $u_t \in \U$, $\phi(x_0, \u) = \{x_t\}_{t \in [k+1]}$ is the sequence of states visited under the control $\u$, where
$x_{t+1} = A_t x_t + B_t u_{t}$, and $ (A_t, B_t) = (A_i, B_i)$ if $x_t \in P_i$.
The cost of the sequence $\phi(x_0, \u)$, $\J(\phi(x_0, \u))$, is given by $\sum_{t \in [k]} \J(x_{t+1}, u_t)$.
We define the partition sequence of $\{x_t\}_{t \in [k+1]}$, denoted $\ps(\{x_t\}_{t \in [k+1]})$, to be the sequence of partitions visited by the states, namely, $P_{i_1}, \ldots, P_{i_{k+1}}$ such that $x_t \in P_{i_t}$ for all $t \in [k+1]$.
\begin{problem}[Optimal control problem]
\label{prob:opt}
$ $\\
Given an $n$-dimensional discrete-time piecewise linear system $\dss = (\X, \Xinit, \U, \P,$ $\{(A_i, B_i, P_i)\}_{i \in [m]}, \Ld, \J)$, a state $x_0^* \in \Xinit$ and a regular property $\Pi$ over $\P$, find a sequence of control inputs $\u^*$ for  which $\Ld(\phi(x_0^*, \u^*)) \in \Pi$ and $\J(\phi(x_0^*, \u^*))$ is minimized.
\end{problem}

\begin{remark}
Although the property $\Pi$ is over finite sequences, it could potentially contain finite sequences of unbounded lengths (i.e. no fixed upper bound on the sequence length).
Hence, the Problem \ref{prob:opt} is not the same as a classical finite horizon problem, because the optimal control sequence length is not fixed a priori. 
\end{remark}

\subsection{Solution}

%
%
%
%

A discrete-time piecewise linear system $\dss = (\X, \Xinit, \U,$ $ \P, \{(A_i,B_i, P_i)\}_{i \in [m]},$ $ \Ld,$ $ \J)$ can be represented as a weighted transition system, $\T_{\dss} = (\X, \Xinit, \U, \P, \Delta, \L, $ $\W)$ where $\Delta = \{(x, u,$ $x') \in \X \times \U \times \X \,|\, x'  = A x + B u, \mbox{ where }$ $ (A, B) = (A_i, B_i) \mbox{ for } x \in P_i\}$, $\L(x) = \Ld(i)$ where $x \in P_i$, and $\W(x, u, x') = \J(x', u)$. Consequently, Problem \ref{prob:opt} is equivalent to the following problem:

\begin{problem}[Optimal strategy problem]
\label{prob:opt_str}
$ $\\
Given a weighted transition system $\T_{\dss} = (\X, \Xinit, \U, \P,\Delta, $ $\L,\W)$, a state $x_0^* \in \Xinit$ and a regular property $\Pi$ over $\P$, find an optimal winning strategy $\sigma(\T_\dss,x_0^*,\Pi)$ for  which the optimal cost of winning $\T_\dss$ with respect to $\Pi$, $\W(\T_\dss, x_0^*, \Pi)$, is achieved.
\end{problem}

\begin{algorithm}[b]
	\caption{$\optcar$ (Abstraction Refinement Procedure)}
	\label{alg:refinement}
	\begin{algorithmic}
		\Require System $\dss$, Property $\Pi$ as a finite state automaton,  initial state  $x_0^*$, rational number $0 < \epsilon_0$
		\State Set $\epsilon := \epsilon_0$
		\While {true}
		    \State $\hat{\T}, \hat{x}_0 := \consabs(\dss, \epsilon)$
			\State $J, \hat{\sigma} := \solve(\hat{\T}, \hat{x}_0, \Pi)$
                        \State $\sigma_\dss := \extract(\hat{\sigma}, \hat{\T}, \dss)$
                        \State Output $\sigma_\dss$ and $J$
			\State $\epsilon := \frac{\epsilon}{2}$
		\EndWhile
	\end{algorithmic}
\end{algorithm}

Note that since $\T_{\dss}$ is input deterministic, $\sigma(\T_\dss,x_0^*,\Pi)$ will correspond to a unique path starting from $x_0^*$.
In general, solving Problem \ref{prob:opt_str} is difficult, since, $\T_\dss$ is a infinite state system; hence, we focus on synthesizing suboptimal strategies using Algorithm \ref{alg:refinement}. As an overview, Algorithm \ref{alg:refinement} first partitions the state space into grids of a particular size, and constructs an abstract system for the system $\T_\dss$. Then, it computes the optimal cost $J$ and strategy of the abstract system through a two-player game. A suboptimal strategy for $\T_\dss$ can then be extracted from the strategy of the abstract system with the cost upper bounded by $J$. If the upper bound $J$ is not zero, refine the state space partitions using smaller grids, and repeat the whole process to reduce the cost $J$. As a result, this algorithm outputs a sequence of suboptimal strategies, whose costs converge to that of the optimal cost.

More precisely, in each iteration, Algorithm \ref{alg:refinement} performs the following sequence of steps. First, it constructs a finite state abstraction $\hat{\T}$ of $\dss$ using the function $\consabs(\dss, \epsilon)$. 
$\consabs(\dss, \epsilon)$ outputs $\abs(\T_\dss, \equiv_X^\epsilon, \equivu^\epsilon)$, where
$\equiv_X^\epsilon$ and $\equivu^\epsilon$ are equivalence relations whose equivalences classes are the elements of $\grid(\X, \epsilon)$ and $\grid(\U, \epsilon)$,  respectively. Define the initial abstract state as $\hat{x}_0 := [x_0^*]_{\equiv_X^\epsilon}$. This step solves $|S|^2 |U|$ optimizations where $|S|$ is the number of states in $\hat{\T}$ and $|U|$ is the number of control inputs in $\hat{\T}$. These optimizations can be computed in parallel. Next, $\solve(\hat{\T}, \hat{x}_0, \Pi)$ computes the optimal cost of winning $J = \W(\hat{\T}, \hat{x}_0, \Pi)$ with respect to $\Pi$ in the finite state transition system $\hat{\T}$ and the corresponding strategy $\hat{\sigma} = \sigma{(\hat{\T}, \hat{x}_0, \Pi)}$ for $\hat\T$ through a two-player game
(see Algorithm \ref{alg:2player} of Section \ref{sec:fts} for more details). 

Finally, $\extract(\hat{\sigma} , \hat{\T}, \dss)$ outputs a suboptimal strategy/controller $\sigma_\dss$ whose cost is bounded by the optimal cost $J$ for the abstract system.
The existence of $\sigma_\dss$ given $\hat{\sigma}$ is guaranteed by Theorem \ref{thm:preserve}. 
Essentially, $\sigma_\dss$ provides the sequence of inputs $u^*$ as required by Problem \ref{prob:opt}. 
To illustrate the relationship between $\sigma_\dss$ and $\hat{\sigma}$, let $u^*_0, u^*_1, \ldots, u^*_{t-1}$ be the inputs which have been computed, and let $s^*_0, s^*_1, \ldots, s^*_{t}$ be the sequence of state generated by the inputs.
The $t$-th control input $u^*_t$ is obtained by finding the minimum cost transition $(s^*_t, u^*_t, s^*_{t+1})$, where $u^*_t \in U$ and $s^*_{t+1} \in S'$. The set $U$ is defined as 
$U = \hat{\sigma}([s^*_0]_{\equiv^\epsilon_X} [u^*_0]_{\equivu^\epsilon} \ldots [s^*_t]_{\equiv^\epsilon_X})$,  and $S'$ is the union of all $S''$ such that 
$([s^*_t]_{\equiv^\epsilon_X}, U, S'')$ is a transition of $\hat{\T}$.
The inputs $u^*_t$ can be computed by solving a linear program when the cost function is linear and the equivalence classes are polyhedral sets.

In the beginning, when the partitioning is coarse, a winning strategy $\hat{\sigma}$ might not exist even if the underlying system $\dss$ has an optimal solution. However, if one continues to refine the grid, a winning strategy will exist if $\dss$ has an optimal solution, and its cost of winning will converge to the optimal cost. See Section \ref{sec:correctness} for the proof. In addition, the algorithm can be terminated at a specific iteration based on applications and computational resources. 

Algorithm \ref{alg:refinement} can in fact be instantiated to any class of hybrid systems.
However, the computational complexity of the optimization problems that will need to be solved in the construction of the abstract system and the extraction of a winning strategy will depend on the class of dynamics and the type of the cost function.
For a piecewise linear system with linear cost function, the maximization during the abstraction procedure is a linear program, because the partitions of $\equiv_X$ and $\equivu$ are polyhedral sets (grid elements). If computation resources are limited, the best suboptimal controller found with respect to the cost upper bound $J$ is guaranteed to generate a trajectory that satisfies the properties $\Pi$ and has cost no greater than $J$. 


\subsection{Analysis of Algorithm \ref{alg:refinement}}
\label{sec:correctness}

This section analyzes the output of Algorithm \ref{alg:refinement}, and shows that the suboptimal cost converges to the optimal cost if a robust optimal strategy exists.
Note that even without the existence of a robust optimal strategy, we can still guarantee that the costs due to refinement are non-increasing.

\begin{definition}
\label{def:robust}
An input sequence $\u$ is said to be \emph{robust} with respect to an initial state $x_0$ if there exists $\epsilon_t > 0$ such that $\B_{\epsilon_t}(x_t) \subseteq P_{i_t}$ for all $t \in [k+1]$, where $\phi(x_0, \u) = \{x_t\}_{t \in [k+1]}$, and $\ps(\phi(x_0, \u)) = \{P_{i_t}\}_{t \in [k+1]}$.
\end{definition}
Let us denote the elements in the iteration of Algorithm \ref{alg:refinement} corresponding to a particular $\epsilon$ as $\hat{\T}_\epsilon$ for $\hat{\T}$, $\hat{x}_0^\epsilon$ for $\hat{x}_0$, $J_\epsilon$ for $J$, $\hat{\sigma}_\epsilon$  for $\hat{\sigma}$ and $\sigma_\epsilon$ for $\sigma_\dss$.

\begin{thm} \label{thm:main_converge}
If there exists a robust optimal control $\u^*$ with respect to $x_0^*$ for Problem \ref{prob:opt}, the sequence of sub-optimal costs $\{J_{\epsilon_0/2^i}\}_{i \in \nnint}$ output by Algorithm \ref{alg:refinement} converges to the optimal cost $J_{\textit{opt}} = \W(\T_\dss,x_0,\Pi)$. Furthermore, for each sub-optimal cost $J_{\epsilon_0/2^i}$, there exists a suboptimal winning strategy $\sigma_{\epsilon_0/2^i}$ with cost of winning $J_{\epsilon_0/2^i}$.
\end{thm}

The rest of this section proves Theorem \ref{thm:main_converge}. 
Proofs of some lemmas are provided in the appendix to improve readability of this section.
Henceforth, let $\u^* = \{u^*_t\}_{t\in[k]}$ be a robust optimal control input sequence with respect ot $x_0^*$ and $\zeta^* = \phi(x_0^*, \u^*) = \{x^*_t\}_{t\in [k+1]}$ be the corresponding optimal trajectory for Problem \ref{prob:opt}. The proof also requires a special kind of strategy that ensures that there is a unique path conforming to this strategy, by choosing inputs that result in exactly one successor state (see Figure \ref{fig:partition}).
\begin{definition}	
	A \textbf{chain strategy} for a transition system $\T$ and an
        initial state $s_0$ is a strategy $\sigma \in \str(\T)$ such
        that there is one path in $\pathm_\sigma(\T, s_0)$.
\end{definition}
\begin{figure}[t]
      \centering
	  \def\svgwidth{\columnwidth}
	  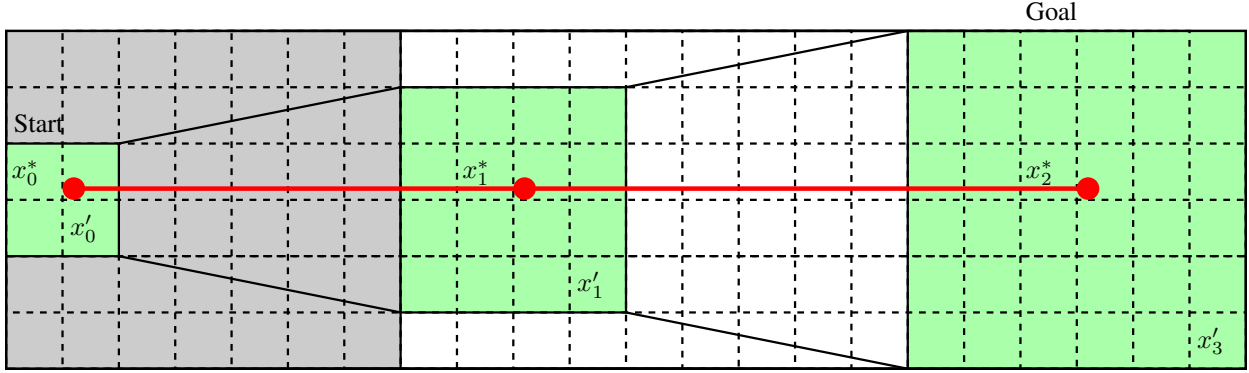
      \caption{An illustration of chain strategy and refinement. The domain is separated into two areas (gray and white) where two different dynamics apply. The red dots are the optimal path.}
      \label{fig:partition}
\end{figure}

To prove Theorem \ref{thm:main_converge}, first, we show that for any trajectory whose initial state and inputs have a bounded deviation from that of the optimal trajectory, the trajectory itself will have a bounded deviation from the optimal trajectory.

\begin{lemma} \label{lemma:xn}
There exist bounds $M_x > 0$ and $M_u > 0$ and constance $c_1, c_2 \geq 0$ that depend on $\ps(\phi(x_0^*, \u^*))$, such that for all $\epsilon_x \in [0,M_x]$ and $\epsilon_u\in [0, M_u]$, if $x_0 \in \B_{\epsilon_{x}}(x^*_0)$ and $u_t\in \B_{\epsilon_{u}}(u^*_t)$ $\forall t \in [k]$, where $\u = \{u_t\}_{t\in[k]}$ and $\phi(x_0, \u) = \{x_t\}_{t \in [k+1]}$, then for all $t \in [k]$,
	\begin{gather*}
		\norm{x_{t+1} - x^*_{t+1}}_{\infty} \leq c_1 \epsilon_{x} +  c_2 \epsilon_{u}, \\
        \ps(\phi(x_0, \u)) = \ps(\phi(x_0^*, \u^*)).
	\end{gather*}
\end{lemma}
This lemma implies that the error from the optimal state at any time is bounded linearly by the error from the initial state and the largest error of control inputs from the optimal ones. As $\epsilon_x$ and $\epsilon_u$ decrease to zero, the state error decreases to zero. Although the constants $c_1$ and $c_2$ depend on $t$, $t$ would not make the constants unbounded because $t$ is finite. Next, we show that the suboptimal cost of this trajectory is bounded.

\begin{lemma}\label{lemma:robust}
Given the cost function in Problem \ref{prob:opt_str}, there exist bounds $M_x > 0$ and $M_u > 0$ and constants $c_3, c_4 \geq 0$ such that for all $\epsilon_x \in [0, M_x]$ and $\epsilon_u \in [0, M_u]$, if $x_0 \in \B_{\epsilon_{x}}(x^*_0)$ and $u_t\in \B_{\epsilon_{u}}(u^*_t)$ $\forall t \in [k]$, where $\u = \{u_t\}_{t \in [k]}$ and $\zeta = \phi(x_0, \u)$,
	\begin{gather*} 
		|\W(\zeta)- \W(\zeta^*)|  \leq c_3 \epsilon_{x} + c_4 \epsilon_{u}, \\
        \ps(\phi(x_0, \u)) = \ps(\phi(x_0^*, \u^*)).
	\end{gather*} 
\end{lemma}
This lemma states that given a continous cost function, there will be a small neighborhood of the optimal trajectory in which the trajectories will go through the same partition sequence and difference in the cost is bounded and decreases to zero if $\epsilon_{x}$ and $\epsilon_u$ decrease to zero.

At this point, we have shown that the suboptimal cost is bounded by terms that depends on the input error and initial state error. Next, we show that given a specific cost sub-optimality, there exists a strategy that satisfies this cost error. In other words, we can construct an abstraction to give a chain strategy that satisfies a certain cost error bound.

\begin{lemma}\label{lemma:chain}
	Given any $\delta > 0 $, there exists a chain winning strategy $\sigma$ for some $\hat{\T} = \abs(\T_\dss, \equiv_X,$ $\equivu)$ such that 
	$$|\W(\hat{\T},\sigma, x'_0) - \W(\T_\dss,x_0^*,\Pi)| \leq \delta$$ 
	where $x'_0 = [x_0^*]_{\equiv_X}$.
\end{lemma}

\begin{IEEEproof}
The broad idea will be to identify neighborhoods $N_t^x$ around $x_t^*$ and $N_t^u$ around $u_t^*$ such that $N_t^x$ is contained in the region of the partition containing $x_t^*$, all transitions from $N_t^x$ and $N_t^u$ lead to $N_{t+1}^x$ and the neighborhoods $N_t^x$ and $N_t^u$ are contained in $B_{M_x}(x_t^*)$ and $B_{M_u}(x_u^*)$.
Further, we will ensure that the maximum cost of any transition from $N_t^x$ to $N_{t+1}^x$ using an input from $N_t^u$ is bounded.
Then, by choosing $N_t^x$ and $N_t^u$ to be regions of $\equiv_X$ and $\equiv_Y$, we obtain a chain strategy in $\hat{\T} = \abs(\T_\dss, \equiv_X,$ $\equivu)$, where the only region of $\hat{\T}$ reachable from the abstract state $N_t^x$ on input $N_t^u$ is $N_{t+1}^x$. 
Refer to Figure \ref{fig:partition} for an illustration of the chain strategy. 

Let $\ps(\{x_t^*\}_t) = \{P_{i_t}\}_t$.
We construct the sequence inductively, starting from $t = k+1$ and moving backwards.
Let $N_{k+1}^x$ be a grid cell of size $\epsilon_0/2^i$ that contains an open ball around $x_{t+1}^*$ which is contained in $P_{i_t}$. We can find such $N_{k+1}$ because of the robustness of the optimal control as defined in Definition \ref{def:robust}.
Assume we have computed $N_{t+1}^x, N_{t+1}^u, \ldots N^x_{k+1}$.
We show how to compute $N_t^x$ and $N_t^u$.
Note that as long as $N_t^{x}$ and $N_t^u$ are contained in $B_{M_x}(x_t^*)$ and $B_{M_u}(x_u^*)$, all the transitions from $N_t^x$ on $N_t^u$ will end in $N_{t+1}^x$. Hence, let $N_t^{x} \subseteq B_{M_x}(x_t^*)$ and $N_t^u \subseteq B_{M_u}(x_u^*)$. By induction, under this construction, all executions from $N_0^x$ will be in $N_t^x$ after $t$ steps. This chain of neighborhoods gives us a chain strategy.

Further, when $N_t^{x} \subseteq B_{\epsilon_x}(x_0^*)$ and $N_t^u \subseteq B_{\epsilon_u}(x_u^*)$ where $\epsilon_x \in [0,M_x]$ and $\epsilon_x \in [0,M_u]$, the cost of the strategy is within $\delta$ of the optimal cost where $c_3\epsilon_x + c_4 \epsilon_u \leq \delta$ for some constants $c_3$ and $c_4$ as given by Lemma \ref{lemma:robust}.
Thus, $|\W(\zeta) - \W(\zeta^*)| \leq \delta$ for any path $\zeta$ starting in an $\epsilon_x$ ball around $x^*_0$.
In addition, choose the $N_j^x$ and $N_j^u$ such that they correspond to an element of an $\epsilon_0/2^i$ grid for some $i$ (not necessarily the same $i$ for all neighborhoods).
Finally, define $\equiv_X$ and $\equivu$ such that the $N_j^x$ and $N_j^u$ are all equivalence classes of $\X$ and $\U$, respectively. Note that we need to ensure that for any $i, j$, $N_j^x$ is the same as $N_i^x$ or the two are disjoint, and, a similar condition for $N_j^u$ holds. This condition can be easily ensured during the construction by picking small enough $\epsilon$.
\end{IEEEproof}
Lemma \ref{lemma:chain} guarantees a chain strategy.
However, the partitions corresponding to the neighborhoods of $N^x $ and $N^u$ may not correspond to an uniform grid for any $\epsilon$.
Enumeration in Algorithm \ref{alg:refinement} only contains uniform grids with grid size $\epsilon_0/2^i$. Thus, 
the next lemma constructs a uniform grid by refining the chain strategy obtained from Lemma \ref{lemma:chain}.

\begin{lemma} \label{lemma:grid}
For a given $\delta >0$, there exists an $\epsilon = \epsilon_0/2^i >0$, such that $|\W(\T_\epsilon,x^\epsilon_0,\Pi) - \W(\T_\dss,x_0^*,\Pi)|$ $\leq \delta$, where $x^\epsilon_0 = [x_0^*]_{\equiv_X^\epsilon}$. Furthermore, there exists a winning strategy $\sigma_\epsilon$ with cost of winning $\W(\T_\epsilon,x^\epsilon_0,\Pi)$.
\end{lemma}

\begin{IEEEproof}
From the proof of Lemma \ref{lemma:chain}, we obtain a sequence of neighborhoods $N_t^x$ and $N_t^u$ which correspond to a chain strategy, say $\sigma$ starting from $N_0^x$.
Further, as observed in the proof, we can assume that every $N_t^x$ corresponds to an element of $\grid(\X, \epsilon_0/2^{i_t})$ for some $i_t$, and similarly, $N_t^u$ corresponds to an element of $\grid(\U, \epsilon_0/2^{j_t})$ for some $j_t$.
Let $i$ be the maximum of the $i_t$s and $j_t$s.
Note that $\grid(\X, \epsilon_0/2^i)$ refines $N_t^x$ and similarly, $\grid(\U, \epsilon_0/2^i)$ refines $N_t^u$.
In Figure \ref{fig:partition}, the squares around $x_t^*$ with bold borders are $N_t^x$, and the dashed squares which are contained in them correspond to the refined partition.
One can define a strategy $\sigma_\epsilon$ (not necessarily a chain anymore) for $\T_\epsilon$ which correspond to following the neighborhoods $N_t^x$.
Hence, all the paths in $\T_\epsilon$ which conform to $\sigma_\epsilon$ will be contained in the neighborhoods $N_t^x$.
Therefore, the cost of $\sigma_\epsilon$ is bounded by that of $\sigma$ which is at most $\delta$ away from the optimal cost.
Therefore, the optimal cost of $\T_\epsilon$ is at most $\delta$ away from that of $\T_\dss$.
\end{IEEEproof}

\noindent\textbf{Proof of Theorem \ref{thm:main_converge}.}
First, observe that $J_{\epsilon_0/2^i} \leq J_{\epsilon_0/2^j}$ for all $i > j$.
Further, from Lemma \ref{lemma:grid}, for any $\delta > 0$, there exists $\epsilon = \epsilon_0/2^i$,  such that $|\W(\T_\epsilon,x^\epsilon_0,\Pi) - \W(\T_\dss,x_0^*,\Pi)|$ $\leq \delta$.
Note $J_\epsilon = \W(\T_\epsilon,x^\epsilon_0,\Pi)$ and $ J_{\textit{opt}} = \W(\T_\dss,x_0^*,\Pi)$ is the optimal cost.
Hence, $|J_\epsilon - J_{\textit{opt}}| \leq \delta$.
Therefore, $J_{\epsilon_0/2^i}$ converges to $J_{\textit{opt}}$ as $i$ goes to infinity.
In addition, from Lemma \ref{lemma:grid}, for each sub-optimal cost $J_{\epsilon_0/2^i}$, there exists a suboptimal winning strategy $\sigma_{\epsilon_0/2^i}$ with cost of winning $J_{\epsilon_0/2^i}$.

At this point, we have shown that the strategy given by $\optcar$ incurs a suboptimal cost that converges to the optimal cost of $\dss$. 
The strategies used in the proof of Theorem \ref{thm:main_converge} have the property that the length of the maximal paths which conform to the strategy are finite and have a bound (in fact, they are all of the same length).
Further, the trace of all the paths is the same. However, during implementation, Algorithm \ref{alg:refinement} may return a sequence of suboptimal strategies $\sigma_{\epsilon_0/2^i}$ that results in paths with different lengths. Nonetheless, the cost of each path results from $\sigma_{\epsilon_0/2^i}$ is bounded by the cost $J_{\epsilon_0/2^i}$.

In addition, the strategy that is considered in the proof of Theorem \ref{thm:main_converge} gives a sequence of inputs which satisfy the property $\Pi$ from any point in an open neighborhood around the given initial state $x_0^*$.
Further, there is an open neighborhood around each of the control inputs such that the resulting paths satisfy $\Pi$.
Hence, Algorithm \ref{alg:refinement} in fact returns a controller that is robust against input uncertainties under the assumption that the original system has such optimal control. 
\section{Optimal Control of Finite Transition Systems} \label{sec:fts}


\begin{algorithm}[b!]
	\caption{$\solve$ (Two-Player Games)}
	\label{alg:2player}
	\begin{algorithmic}
		\Require Finite state transition system $\T_S$, Property $\Pi$ specified as $(\T_P, P_f)$
                \State $\T, S_f := \redreach(\T_S, \T_P, P_f)$ 
                \State Set for every $s \in \S - S_f$, $\cost(s) := 0$ if $s \in S_f$ and $\infty$ otherwise
                \For {$i = 1,\ldots,|\S|$}
			\For {$s \in \S$}
\State \[\cost^i(s) := \min_{u \in \U} \max_{(s, u, s') \in \Delta}
(\W(s, u, s') + \cost^{i-1}(s'))\]
\State \[\sigma^i(s) := \argmin_{u \in \U} \max_{(s, u, s') \in
  \Delta}  (\W(s, u, s') + \cost^{i-1}(s'))\]
			\EndFor
		\EndFor	
		\If  {$\cost^{|S|}(s_0) < \infty$}	 
                        \State Output the strategy $\sigma^{|S|}$ and the cost $C^{|S|}(s_0)$
		\EndIf
	\end{algorithmic}
\end{algorithm}

This section presents a value iteration scheme for computing the optimal cost and optimal strategy for finite transition systems.
Observe that the strategies of the abstract system that are used in the proof of Theorem \ref{thm:main_converge} have a linear structure, that is, there are no paths in the abstract system of length greater than the number of states in the system that conform with the strategy.
We call such a strategy a layered strategy.
Hence, in this section we present an algorithm for computing an optimal strategy for a finite state transition system that is layered.
The algorithm is given in Algorithm \ref{alg:2player} which is a modified Bellman-Ford algorithm \cite{bellman1956routing}.

The function $\redreach$ reduces the problem of computing the layered strategy for a property $\Pi$ to that of reachability.
It consists of taking a product of the input transition system $\T_S$ and the transition system $\T_P$ of the property. More formally, given the input transition system $\T_S$ and the transition system $\T_P$ of the property, the product transition system returned by $\redreach$ is defined as follows.
\begin{definition}
Let $\T_S = (\S_S, \Sinit_S, \U_S, \P, \Delta_S, \L_S,
\W_S)$ be a state transition system, and $\T_P = (\S_P, \Sinit_P,
\U_P, \P,\Delta_P, $ $\L_P, \W_P)$ be the automation that represents the regular property. Then, the product transition system is
$\T = (\S, \Sinit,
\U, \P, \Delta, \L, \W)$ where
\begin{itemize}
\item $\S = \{(s_1, s_2) \in \S_S \times
\S_P \,|\, \L_S(s_1) = \L_P(s_2)\} \cup \{s_d\}$ where $s_d$ is a dead state; 
\item $\Sinit =  \Sinit_S \times \Sinit_P$;
\item $\U = \U_S$;
\item $\P$ is the same for both $\T_S$ and $\T_P$. The final states of $\T_P$ is denoted by a proposition $P_f \in \P$;
\item $\Delta = \Delta_1
\cup \Delta_2$, where $\Delta_1 = \{ ((s_1, s_2), u, (s'_1, s'_2)) \in
\S \times \U \times (\S \backslash \{s_d\}) \,|\, (s_1, u, s'_1) \in \Delta_S, (s_2, a,
s'_2) \in \Delta_P$ for some $a\}$ and $\Delta_2=\{((s_1, s_2), u, s_d) \in \S \times U \times \{s_d\}\}$ such that there
exists $(s_1, u, s'_1) \in \Delta_S$ for some $s'_1$ and there does
not exist $a$ and $s'_2$ such that $(s_2, a, s'_2) \in \Delta_P$ and
$\L_S(s'_1) = \L_P(s'_2)$;
\item $\L(s) = \L_S(s)$ for $s \in \S_S$;
\item $\W((s_1,s_2),u,(s'_1,s'_2)) = \W_S(s_1,u,s_2)$.
\end{itemize}
\end{definition}
Furthermore, the set of final states $S_f$ of $\T$ with respect with reachability is solved as $S_f = \{(s_1, s_2) \in (\S \backslash \{s_d\}) \times
(\S \backslash \{s_d\}) \,|\, \L_P(s_2) = P_f\}$.

The algorithm initially assigns a cost of $0$ to the states in $S_f$ and $\infty$ otherwise.
The cost $\cost^i$ in the $i$-iteration captures the optimal cost of reaching $S_f$ by a strategy in which all paths that conform to it have length at most $i$, and $\sigma^i$ stores a corresponding strategy.
Hence, $\cost^{|\S|}$ provides a layered strategy if $\cost^{|\S|}(s_0)
< \infty$.
The algorithm can be improved wherein it terminates earlier than
completing the $|\S|$ iterations, if the costs $\C$ do not change
between iterations. In the worst case, this algorithm runs in $O(|\Delta||\S|)$ time where $|\Delta|$ is the number of transitions in $\T$ and $|\S|$ is the number of states in $\T$.

\section{Implementation} \label{sec:example}

Algorithm \ref{alg:refinement} and \ref{alg:2player} are implemented
in the tool $\optcar$ in Python 2.7.
A Python package, NetworkX, is used to represent the graph structures
that arise in solving Algorithm \ref{alg:2player}, and the Parma Polyhedra
Library \cite{BagnaraHZ08SCP} is used to represent the polyhedral sets
that arise in the gridding and to solve the linear program problem
that arises in the weight computation.
$\optcar$ is tested on 
a linear dynamical system and a piecewise linear
system on a MacBook Pro 8.2, 4 core Intel Core i7 processor
with speed 2200 Hz, and 8GB RAM.

\subsection{Linear Dynamical System}

The following linear dynamical system example is obtained from \cite{tazaki_abstraction_2012}:
\begin{gather}
	x_{t+1} = Ax_t+Bu_t \label{eqn:linear_system}\\
	A = \left[\begin{array}{cc}
	0.68 & -0.14 \\
	0.14 & 0.68
	\end{array}\right] \quad
	B = \left[\begin{array}{c}
	0 \\
	0.1
	\end{array}\right] \nonumber
\end{gather}
where $x_t = (x_t^1,x_t^2) \in [-1,1]^2$, and $u_t\in [-1,1]$.

\begin{figure}[b!]
      \centering
      \includegraphics[trim = 3.2cm 6cm 4cm 6.5cm, clip, width=0.45\textwidth]{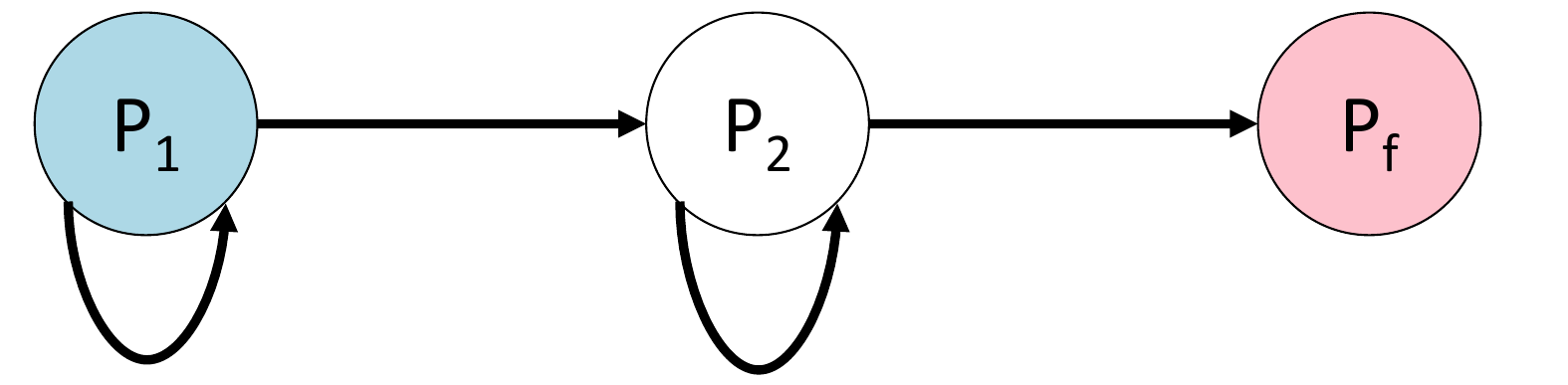}
	  \caption{Automaton that represents the propositions of the two examples where $P_1$ is the pink region, $P_2$ is the white region, and $P_f$ is the light blue region.}\label{fig:proposition_example_1}
\end{figure}

The cost function is $\J(\phi(x_0,u)) =\sum_{t\in[k]} \norm{u_t}^2_2$. This cost is approximated as $\sum_{t\in[k]} \norm{u_t}_1$ during implementation of $\optcar$. The goal is to drive the system from an initial point $x_0 = (0.9,0.9)$ to a final zone defined by a box at the origin, $P_f = \{x \mid\norm{x}_{\infty} \leq \frac{1}{5}\}$. The propositions of this example are represented in Figure \ref{fig:proposition_example_1}. The algorithm is implemented on two uniform grids on the states - $20 \times 20$ and $40 \times 40$. The input, $u$, is partitioned into 5 uniform intervals.

\begin{table*}[t]
\renewcommand{\arraystretch}{1.3}
\centering
\begin{threeparttable}
	\caption{Performance of $\optcar$ and LQR for System \eqref{eqn:linear_system}. \tnote{1} }
	\label{tab:result_linear}
	\begin{tabular}{l|cc|c}\hline
	Grid & $20\times 20$ & $40\times 40$ & LQR\\
	\hline
	Computation time (seconds) & 355.82 & 5212.91 & 0.04 \\
	Optimal cost & 0.5 & 0 & 0 \\
	Optimal step & 6 & 6 & 6 \\
	Final point & (-0.0468,0.1499) &(-0.0468,0.1999)&(-0.0468,0.1999) \\
	\hline
	\end{tabular}
    \begin{tablenotes}
\item [1] $20 \times 20$ and $40 \times 40$ are optimal controller synthesis using $\optcar$ for two different uniform grids. Computation time is the time a method takes to compute the optimal strategy. Optimal step is the total number of steps that the optimal path takes to reach the goal region. Final point is where the optimal path ends in the goal region.
\end{tablenotes}
    \end{threeparttable}
\end{table*}

Strategies obtained from $\optcar$ are compared with the strategy given by linear quadratic regulator (LQR) in Table \ref{tab:result_linear}. For a linear system, LQR is always a superior technique in comparison to $\optcar$ because the computation is significantly more efficient. The goal of this example is to illustrate that in an example with known optimal controller, the strategies given by $\optcar$ approximates the optimal control of LQR very closely, and it improves with refinement. Figure \ref{fig:path_linear-traj} shows the state trajectory for the three cases.

\begin{figure}[b!]
      \centering
{\includegraphics[trim = 15mm 0mm 25mm 15mm, clip, width=0.6\linewidth]{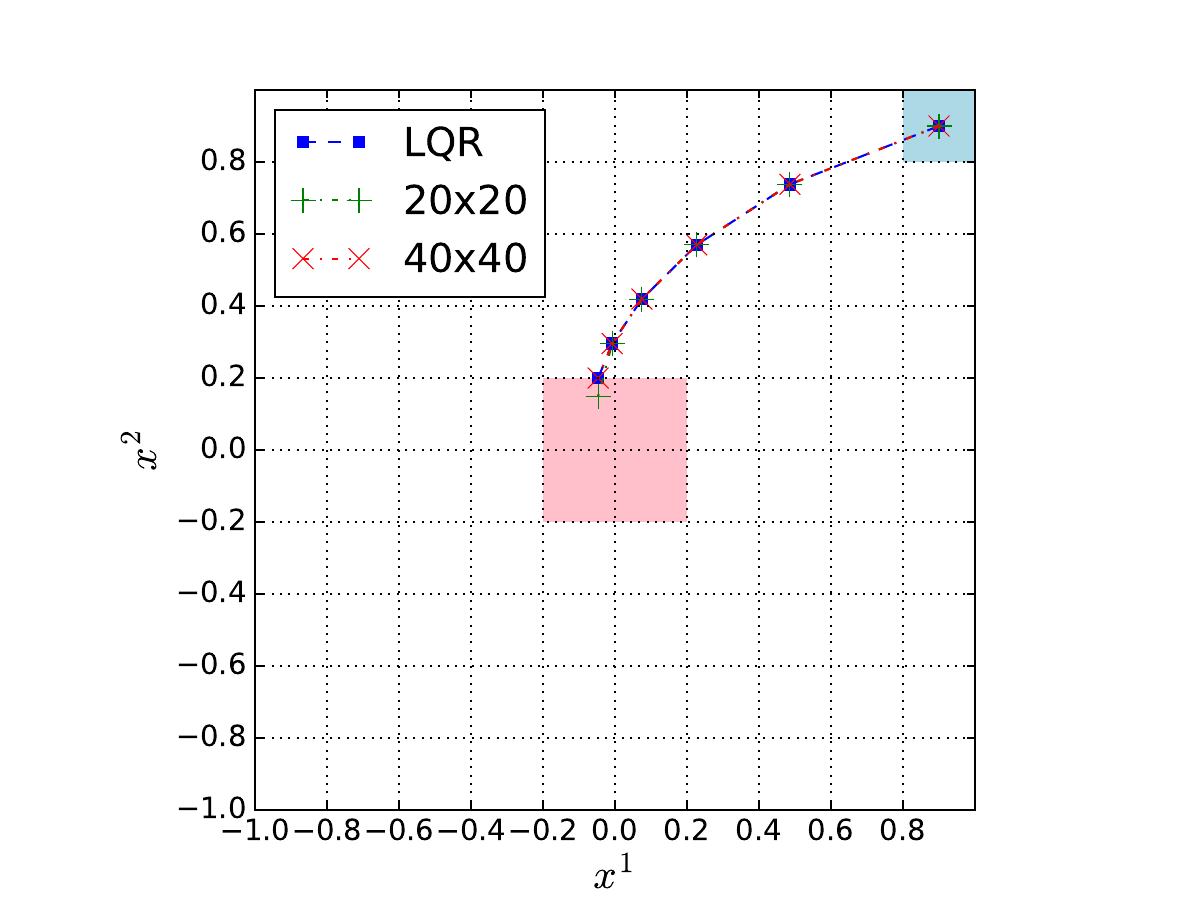}
	  \caption{Simulated result of $\optcar$ and LQR on the linear dynamical system. }\label{fig:path_linear-traj}}
\end{figure}

\begin{figure}[b]
      \centering
	  {\includegraphics[width=0.6\linewidth]{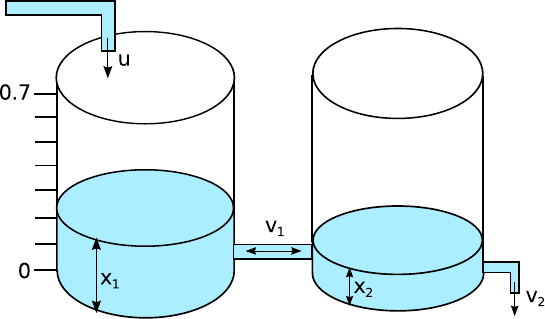}}
      \caption{A schematic of a two-tank system. }
      \label{fig:watertank}
\end{figure}

\subsection{Two-tank System}

\begin{figure*}[t]
      \centering
	  \subfloat[Uniform $28\times17$]{\includegraphics[trim = 30mm 10mm 25mm 10mm, clip=true, width=0.25\linewidth]{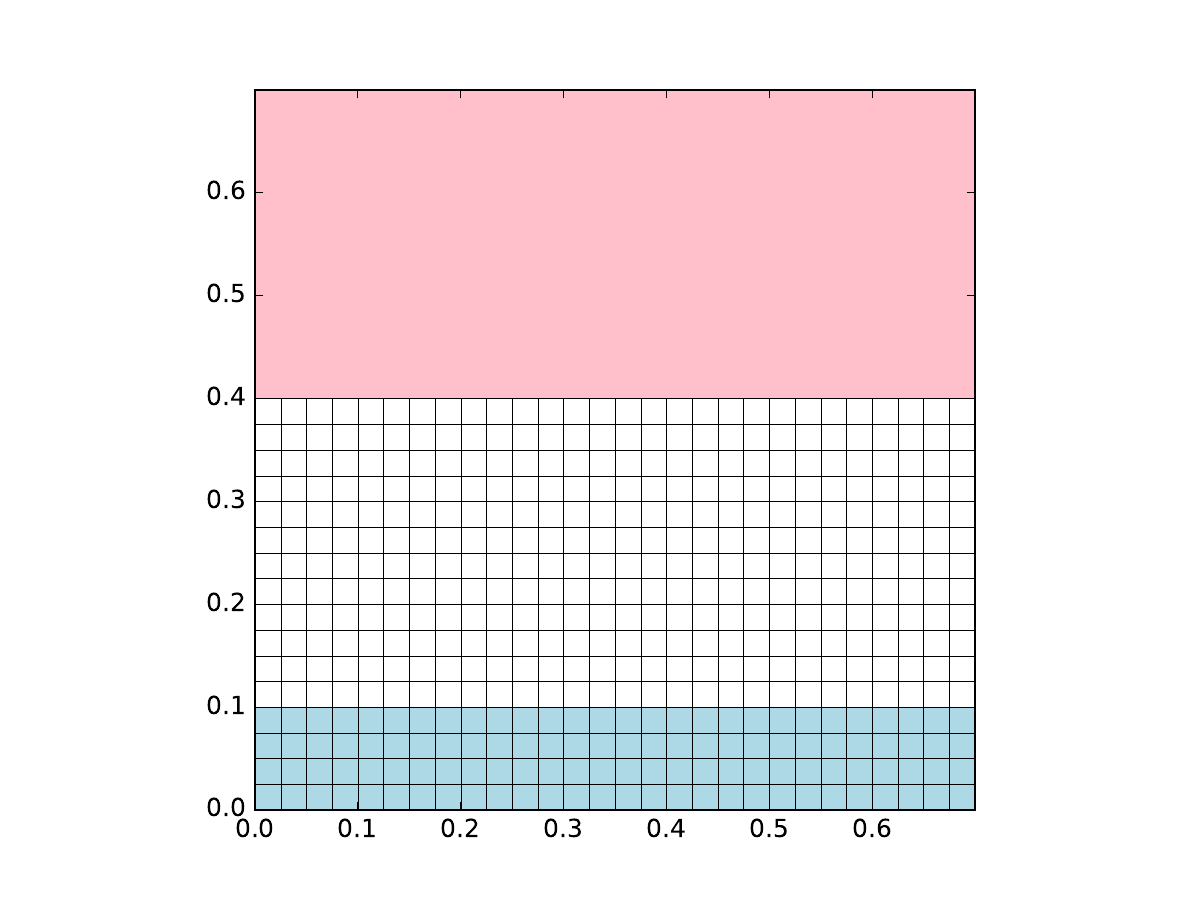}}
       \subfloat[Uniform $56\times33$]{\includegraphics[trim = 30mm 10mm 25mm 10mm, clip=true, width=0.25\linewidth]{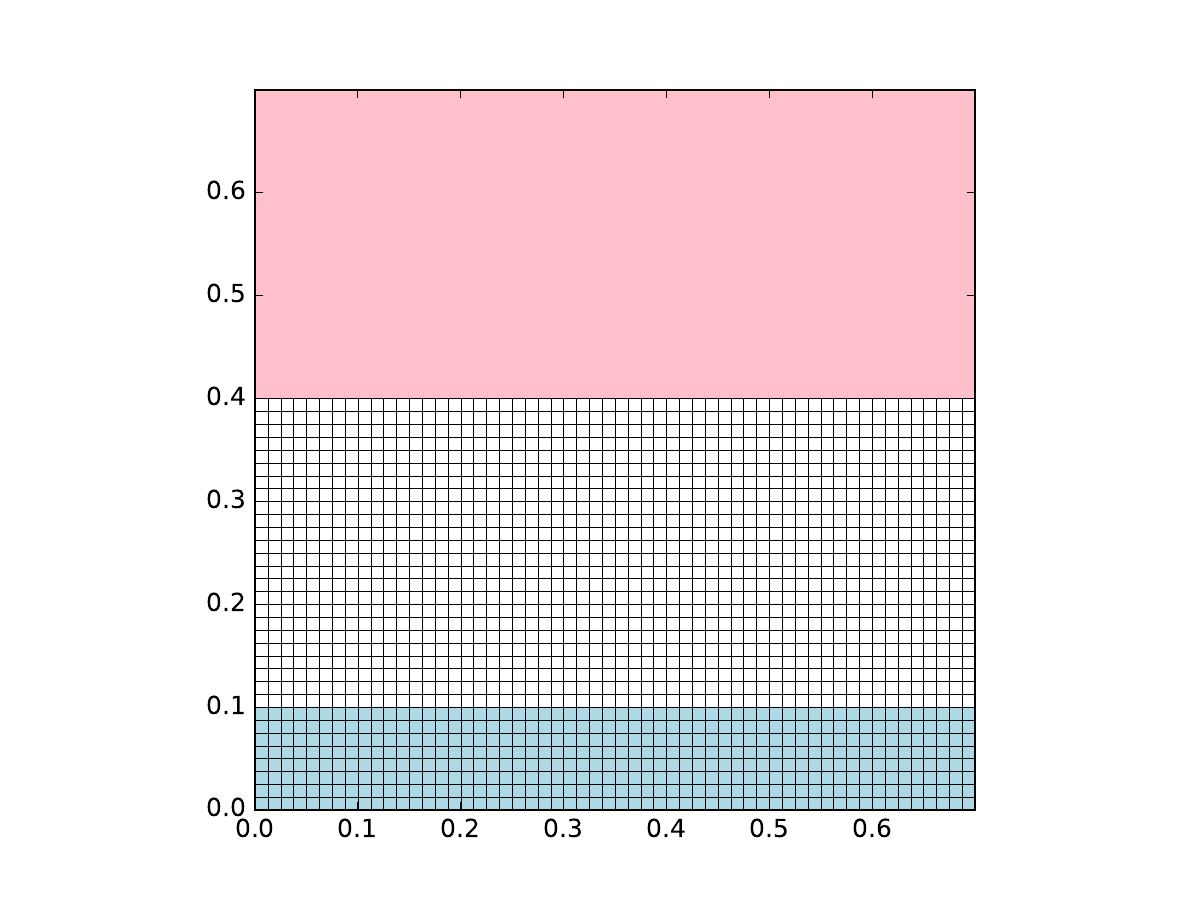}}
       \subfloat[Non-uniform $23\times17$]{\includegraphics[trim = 30mm 10mm 25mm 10mm, clip=true, width=0.25\linewidth]{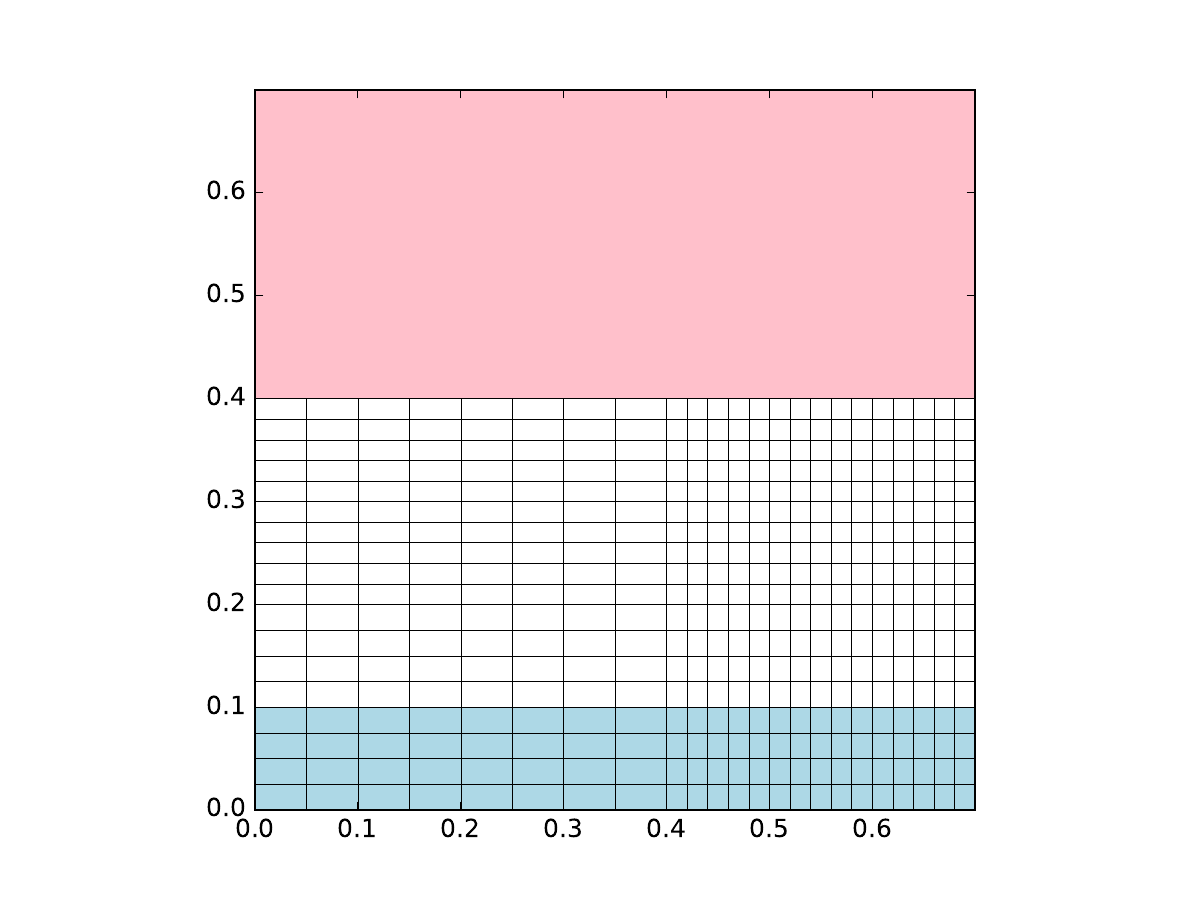}}
       \subfloat[Non-uniform $38\times25$]{\includegraphics[trim = 30mm 10mm 25mm 10mm, clip=true, width=0.25\linewidth]{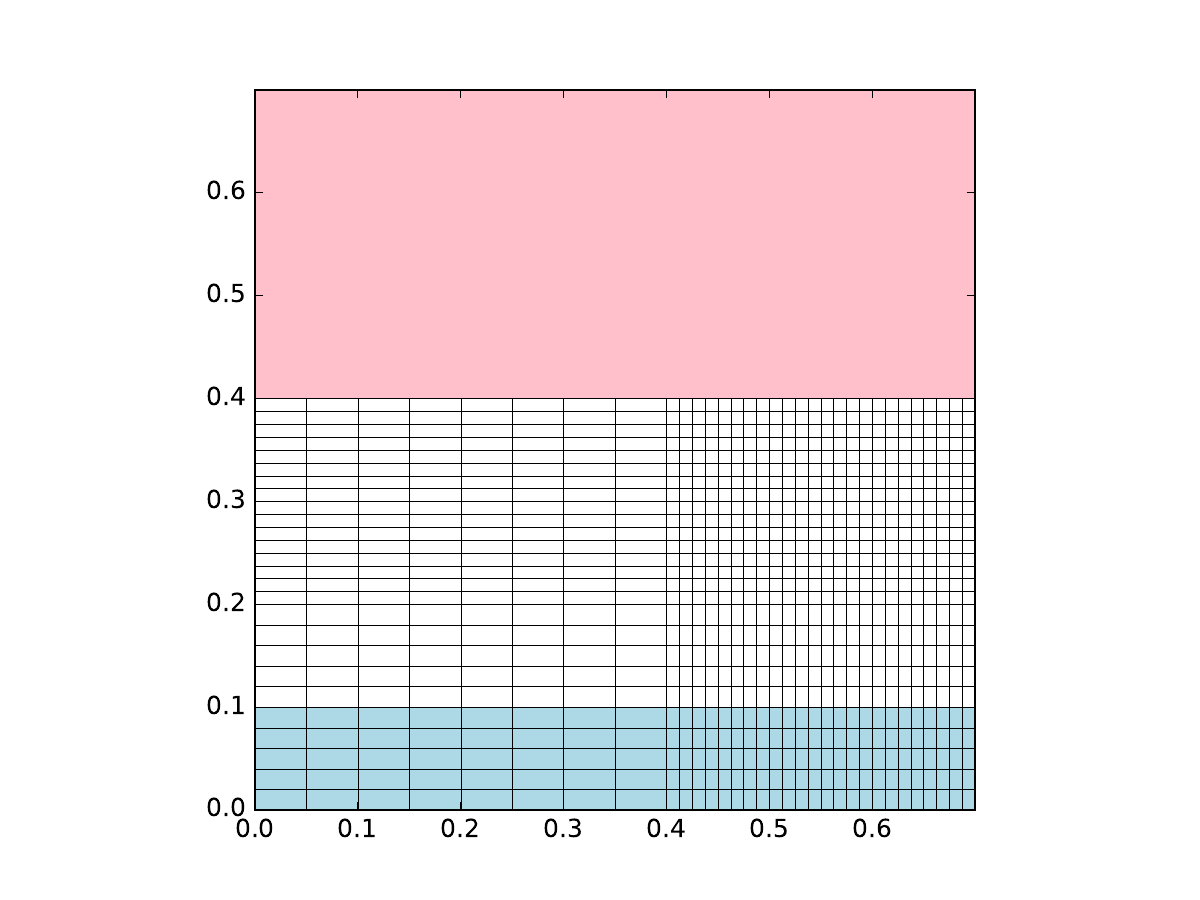}}
      \caption{Partitions for the two-tank system. The goal region does not need to be partitioned because the transitions within the goal region are irrelevant.}
      \label{fig:watertank_grid}
\end{figure*} 

A two-tank system from \cite{yordanov_temporal_2012} is used as an
example of a piecewise linear system (Figure \ref{fig:watertank}). The water can flow in between the two tanks through a pipe that connects them. The pipe is located at level 0.2. Tank 1 (left) has an inflow of water that is managed by a controller, and tank 2 (right) has an outflow of water that is fixed. The goal of the controller is to fill up tank 2 to level 0.4 from an initially low water level 0.1 using as small amount of water as possible from the source above tank 1. The goal will be made precise after the system is described formally next. 

The two-tank system has a linearized dynamics given by 
\begin{gather}
	x_{t+1} = A x_t+Bu_t \label{eqn:two_tank}\\
	A = \left\{\begin{array}{cr}
	A_1 &  x \in [0,0.2]^2\\
	A_2 & \mbox{ otherwise} 
	\end{array}\right. \quad
	B = \left[\begin{array}{c}
	342.6753 \\
	0 
	\end{array}\right],\nonumber
\end{gather}
where 
\begin{gather*}	
	A_1 = \left[\begin{array}{cc}
	1 & 0 \\
	0 & 0.9635 
	\end{array}\right] \quad
	A_2 = \left[\begin{array}{cc}
	0.8281 & 0.1719 \\
	0.1719 & 0.7196 
	\end{array}\right],
\end{gather*}
$x_t = (x^1_t,x^2_t) \in [0,0.7]^2$, and $u_t\in [0,0.0005]$. The water level in tank 1 at time $t$ is $x^1_t$, and the water level in tank 2 at time $t$ is $x^2_t$. The cost function is chosen to be $\J(\phi(x_0,u)) = \sum_{t\in[k]} \norm{u_t}_1$ to represent minimal water inflow, and the goal is to drive the system from partition, $[0,0.7] \times [0,0.1]$, to partition $[0,0.7] \times [0.4,0.7]$. The propositions of this example are represented in Figure \ref{fig:proposition_example_1}.

\begin{figure}[b!]
      \centering
      \subfloat[State trajectory]{\includegraphics[trim = 25mm 0mm 25mm 12mm, clip, width=0.5\linewidth]{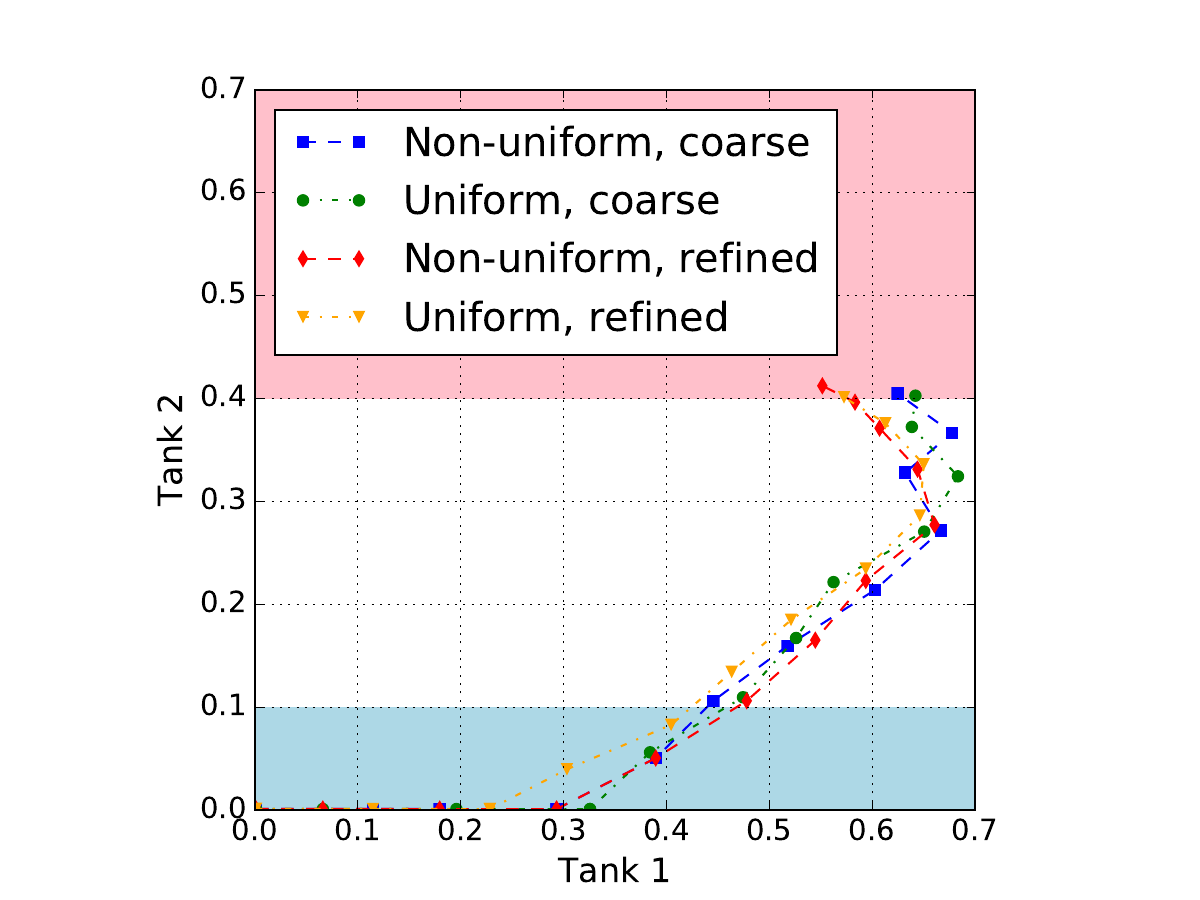}}
      \subfloat{\includegraphics[trim = 3mm 0mm 20mm 10mm, clip, width=0.5\linewidth]{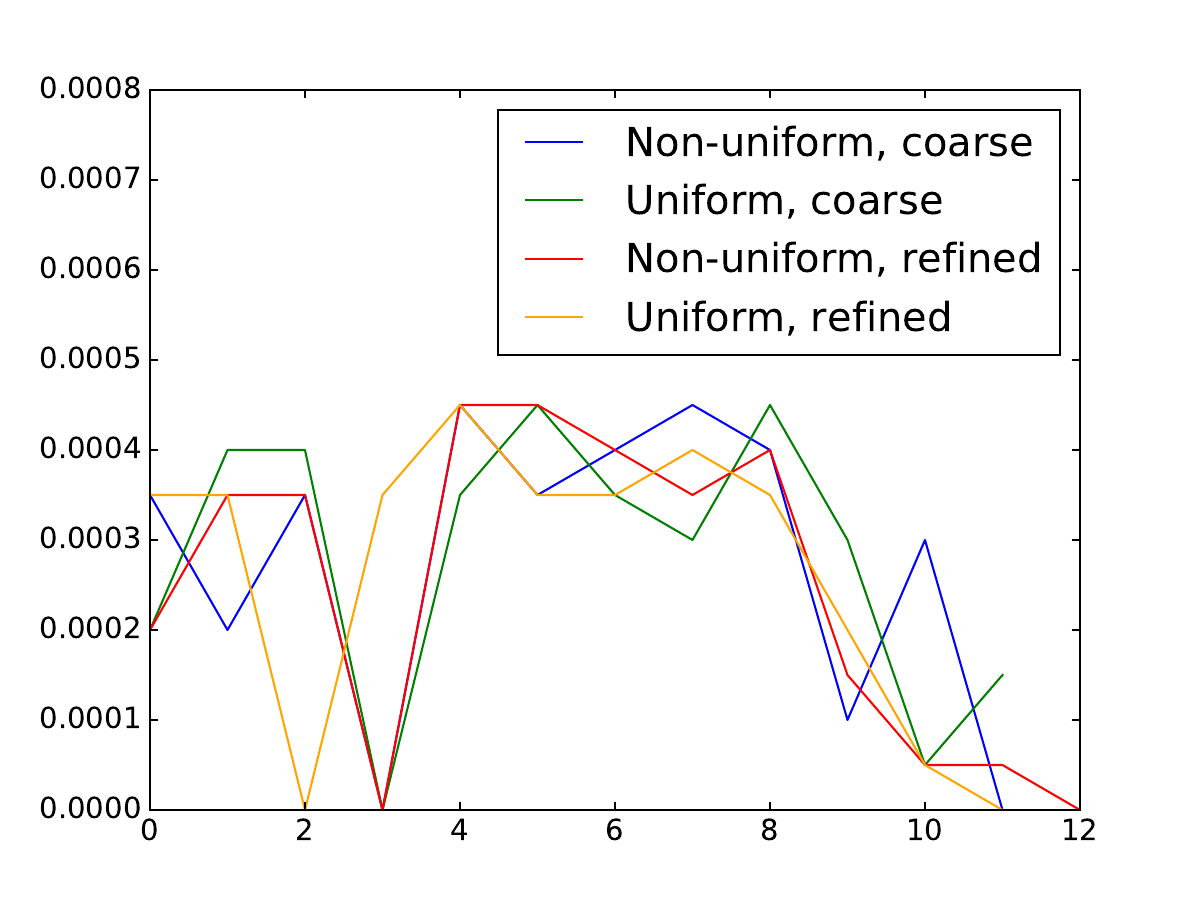}}
      \caption[Control input]{State trajectory and control input generated by the controller from $\optcar$ for the two-tank system. }\label{fig:path_watertank-traj}
\end{figure}


The algorithm is implemented on two uniform grids on the states - $28\times 17$ (coarse) and $56\times 33$ (refined), and two non-uniform grids - $23\times 17$ (coarse) and $32\times 25$ (refined). Figure \ref{fig:watertank_grid} illustrates the grids. The input, $u$, is partitioned into 10 uniform intervals. The goal region is represented as one partition for all cases. This choice of goal representation speeds up computation time, and does not change the results in Section \ref{sec:correctness}. Once a path arrives at the goal region, the path ends. Thus, the goal region does not need to be partitioned because the transitions within the goal region is irrelevant. In addition, the partitions' sizes for a non-uniform grid do not necessary have to be the same. Partitions whereby the transitions are more likely to be far can be larger because the states most likely will not end up at the neighboring partitions if the partitions are small. Another example of non-uniform grids with the same principle would be to have finer grids near the goal and coarser grids away from the goal. Such modification is feasible if the control engineer has prior information about the system from his/her past experiences. These modifications reduce computation time, and also allow for finer grids at regions that matter to achieve a better result. 

Strategies obtained from $\optcar$ are compared in Table \ref{tab:result_watertank}. Figure \ref{fig:path_watertank-traj} shows the state trajectory and control input for the four cases, all start from $(0.001, 0.001)$. This example shows that choosing a suitable partition can reduce the computation time dramatically while still achieving comparable performance to the performance of a naive uniform grid. Hence, future extension of this technique includes designing an intelligent scheme to partition the domain such that computation time is reduced. Lastly, about $60\% - 70\%$ of the computation time are used to construct the abstraction (i.e. $\consabs$ step in Algorithm \ref{alg:refinement}) in which the computations can be parallelized easily to decrease computation time.

\begin{table*}[t]
\renewcommand{\arraystretch}{1.3}
	\centering	
    \begin{threeparttable}
	\caption{Performance of $\optcar$ using Different Grids for System \eqref{eqn:two_tank}. \tnote{2}}
	\label{tab:result_watertank}
	\begin{tabular}{l|cc|cc}\hline
	Grid & $28\times 17$ & $56\times 33$ & $23\times 17$ & $38\times 25$  \\
	\hline
	Computation time (seconds) & 1234.53 & 22119.36 & 1057.43 & 4309.04\\
	Optimal cost & 0.00340 & 0.00320 & 0.00335 & 0.00320 \\
	Optimal step & 12 & 12 & 12 & 13 \\
	Final point & (0.642,0.402) &(0.573,0.401) & (0.625,0.405) & (0.552,0.412) \\
	\hline
	\end{tabular}
    \begin{tablenotes}
\item [2] The first two columns are results for uniform grids. The last two columns are results for non-uniform grids. The grids are shown in Figure \ref{fig:watertank_grid}. Computation time is the time $\optcar$ takes to compute the optimal strategy. Optimal step is the total number of steps that the optimal path takes to reach the goal region. Final point is where the optimal path ends in the goal region.
\end{tablenotes}
    \end{threeparttable}
\end{table*}

\section{Conclusion} \label{sec:conclusion}
In this paper, we consider the problem of synthesizing optimal control
strategies for discrete-time piecewise linear system with respect to
regular properties.
We present an abstraction-refinement approach for constructing
arbitrarily precise approximations of the optimal cost and the
corresponding strategies. This approach computes a sequence of suboptimal controller that converges to the optimal controller with refinement. The resulting suboptimal controller would generate trajectories that incurs cost no greater than the optimal cost of the corresponding abstract system.
The abstraction based approach can be applied to the general class of
hybrid systems and for properties over infinite traces, however, the
challenge is in computing edges and weights, especially, for
non-linear dynamics and in continuous time. 

Future work will include extending the technique to more complex dynamics and continuous-time hybrid systems. In addition, the cost preserving abstraction technique will be extended from regular properties to $\omega$-regular properties. To reduce computation time, a more intelligent gridding scheme in the refinement step will be developed. Lastly, the neighborhoods of states and inputs in the abstract system naturally model measurement errors and input uncertainties of the concrete system. Hence, a potential future application of this technique is in synthesizing robust optimal control for a hybrid system.


\section*{Acknowledgment}
Pavithra Prabhakar was partially supported by EU FP7 Marie Curie Career Integration Grant No. 631622 and NSF CAREER 1552668. This work was partially conducted when Yoke Peng Leong was an intern at IMDEA Software Institute.

\appendices

\section{Proof for Lemma 7}

Rewrite $x_{t+1} = A_t x_t+B_t u_t$ as 
\begin{align*}
	x_{t+1} = \left(\prod^t_{j=0} A_j \right) x_0+B_t u_t + \sum_{k=1}^{t}\left( \left(\prod^t_{j=k} A_j \right) B_{k-1}u_{k-1}\right)
\end{align*}
where $\prod^t_{j=0} A_j  = A_t A_{t-1} \ldots A_{1} A_{0} $.
Then,
\begin{align*}
	&\norm{x_{t+1} - x^*_{t+1}}_{\infty}\\
	&\quad= \left|\!\left| \left(\prod^t_{j=k} A_j \right)(x_0-x^*_0)+B_t (u_t-u^*_t) +\sum_{k=0}^{t-1} \left( \left(\prod^t_{j=k} A_j \right) B_{k-1} (u_{k-1}-u^*_{k-1})\right)\right|\!\right|_{\infty}\\
	&\quad\leq \prod^t_{j=0} \norm{A_j }_{\infty} \norm{x_0-x^*_0}_{\infty} + \norm{B_t}_\infty \norm{u_t-u^*_t}_\infty  + \sum_{k=1}^{t}\left(\prod^t_{j=k} \norm{A_j}_\infty \right)\norm{B_{k-1}}_{\infty} \norm{u_{k-1}-u^*_{k-1}}_{\infty}\\
	&\quad\leq \prod^t_{j=0} \norm{A_j }_{\infty}  \epsilon_{x} + \norm{B_t}_\infty  \epsilon_{u}  + \left( \sum_{k=1}^{t}\left(\prod^t_{j=k} \norm{A_j}_\infty \right) \norm{B_{k-1}}_{\infty}\right)  \epsilon_{u} \\
	&\quad=c_1 \epsilon_{x} +  c_2 \epsilon_{u}
\end{align*}
where  $c_1 = \prod^t_{j=0} \norm{A_j }_{\infty} $ and $c_2 = \norm{B_t}_\infty+\sum_{k=1}^{t}\left(\prod^t_{j=k}\right. $ $\left.\norm{A_j}_\infty \right) \norm{B_{k-1}}_{\infty}$.

\section{Proof for Lemma 8}

First, compute
\begin{align*}
	 |\W(\zeta)- \W(\zeta^*)| &=  \left|\sum^k_{t=0} \J(x_{t+1}, u_t) - \J(x_{t+1}^*, u_t^*)\right| \\
	 & \leq \sum^k_{t=0}| \J(x_{t+1}, u_t) - \J(x_{t+1}^*, u_t^*)|. 
\end{align*}
Since $\J(x,u)$ is a continuous function, there exists a constant $C_t > 0$ such that 
\begin{equation*}
	| \J(x_{t+1}, u_t) - \J(x_{t+1}^*, u_t^*)| \le C_t \norm{ \bar{x}_t - \bar{x}_t^*}_\infty. 
\end{equation*}
Hence,
\begin{align*}
	 &|\W(\zeta)- \W(\zeta^*)| \leq \sum^k_{t=0} C \norm{ \bar{x}_t - \bar{x}_t^*}_\infty 
\end{align*}
where $C = \max_{t\in [k]} C_t$ and $\bar{x}_t = [x_{t+1},u_t] \in \reals^{n+p}$ is a joined vector of $x$ and $u$.
By Lemma \ref{lemma:xn}, 
\begin{align*}
	 |\W(\zeta)- \W(\zeta^*)| & \leq \sum^k_{t=0} C \max\{c_1(t+1) \epsilon_x + c_2(t+1) \epsilon_u, \epsilon_u\} \\
	 &\le \max\{c'_3 \epsilon_x + c'_4  \epsilon_u, k C \epsilon_u\}
\end{align*}
where $c'_3 = \max_{t\in [k]} kCc_1(t+1)$, and $c'_4 = \max_{t\in [k]} k C c_2(t+1)$.
Then, 
\begin{align*}
	 & |\W(\zeta)- \W(\zeta^*)| \le c_3 \epsilon_x + c_4 \epsilon_u \\
	 & c_3 = \left\{
	 \begin{array}{ll} c'_3, &  c'_3 \epsilon_x + c'_4  \epsilon_u \ge kC\epsilon_u \\
	 0, &  c'_3 \epsilon_x + c'_4  \epsilon_u < kC\epsilon_u
	 \end{array} \right. \\
	 & c_4 = \left\{\begin{array}{ll} c'_4, &  c'_3 \epsilon_x + c'_4  \epsilon_u \ge kC\epsilon_u \\
	 kC, &  c'_3 \epsilon_x + c'_4  \epsilon_u < kC\epsilon_u
	 \end{array}\right.
\end{align*}

	


\bibliographystyle{IEEEtran}
\bibliography{OptCAR}

%

\begin{IEEEbiography}[{\includegraphics[width=1in,height=1.25in,clip,keepaspectratio]{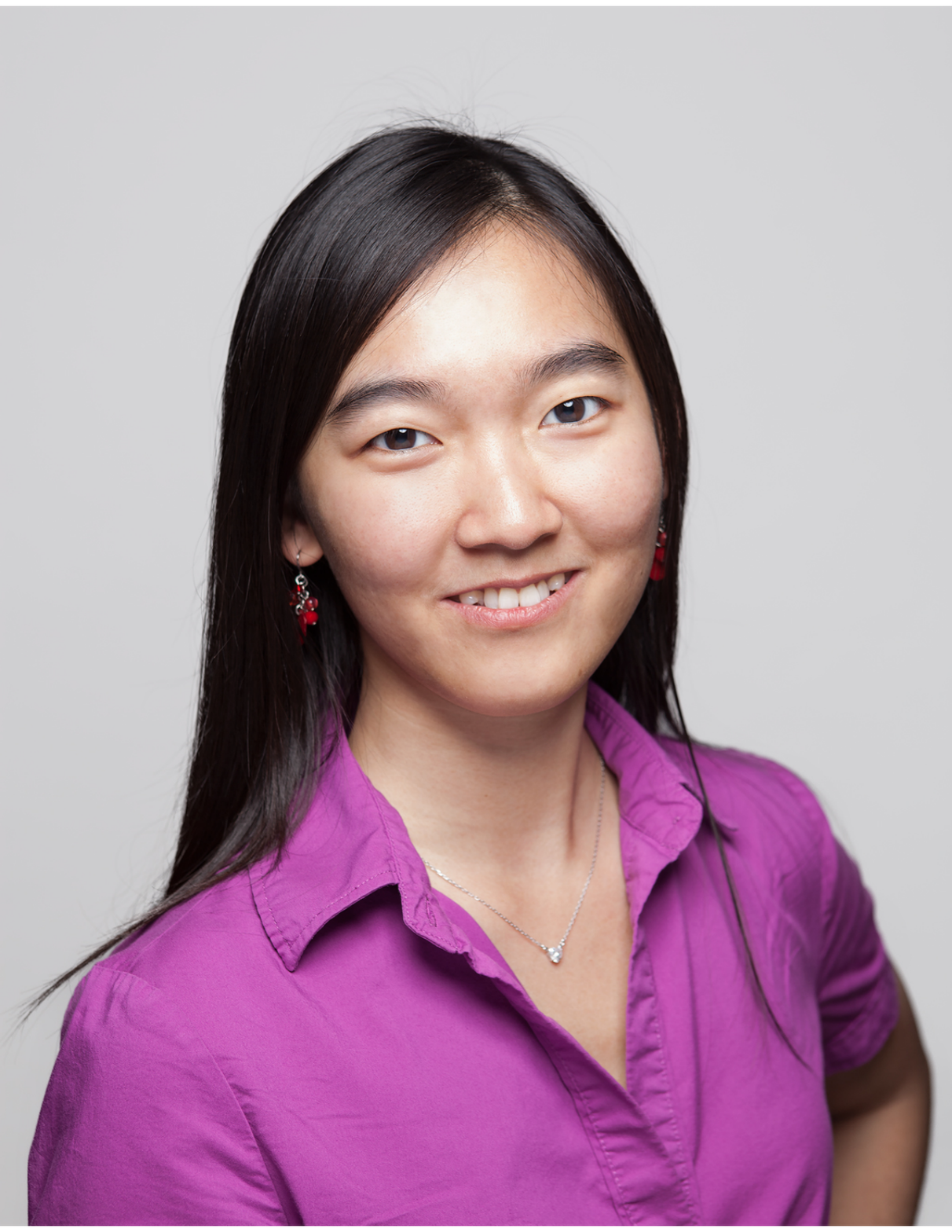}}]{Yoke Peng Leong}
received the B.S. and M.S. degrees
in Mechanical Engineering from Northwestern University, Evanston, IL, USA, in 2012. She
is currently a Ph.D. candidate in Control and
Dynamical Systems at the California Institute of Technology, Pasadena, CA, USA.
\end{IEEEbiography}

\begin{IEEEbiography}[{\includegraphics[width=1in,height=1.25in,trim=0cm 20px 0cm 20px,clip,keepaspectratio]{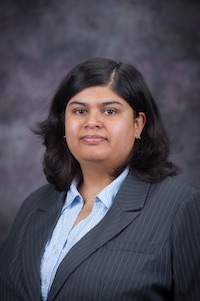}}]{Pavithra Prabhakar}
is an associate professor of computer science at Kansas State University, where she holds the Peggy and Gary Edwards Chair in Engineering. She obtained her doctorate in Computer Science from the University of Illinois at Urbana-Champaign (UIUC) in 2011, from where she also obtained a masters in Applied Mathematics. She was a CMI (Center for Mathematics of Information) fellow at Caltech for the year 2011-12. She has been on the faculty of Kansas State University since 2015, and has previously held a faculty position at the IMDEA Software Institute. Her main research interest is in the Formal Analysis of Cyber-Physical Systems, with emphasis on both theoretical and practical methods for verification and synthesis of hybrid control systems. Her papers have been selected for a best paper honorable mention award from Hybrid Systems: Computation and Control, best papers of MEMOCODE and invited papers at Allerton and American Control Conference. She has been awarded a Sohaib and Sara Abbasi fellowship from UIUC, an M.N.S Swamy medal from the Indian Institute of Science for the best masters thesis, a Marie Curie Career Integration Grant from the European Union, Michelle Munson-Serban Simu Keystone Research Faculty Scholarship from the KSU College of Engineering, a summer faculty fellowship from AFRL, an NSF CAREER Award and an ONR Young Investigator Award.
\end{IEEEbiography}





\end{document}